\documentclass[a4paper,11pt]{article}
\usepackage[margin=2.0cm]{geometry}
\usepackage{times}
\usepackage[latin1]{inputenc}
\usepackage[english]{babel}
\usepackage{amsthm}
\usepackage{amsmath}
\usepackage{amsfonts}
\usepackage{indentfirst}
\usepackage{enumerate}
\usepackage{cancel}
\usepackage{stmaryrd}
\theoremstyle{plain}
\newtheorem{thm}{Theorem}[section]
\newtheorem{prop}[thm]{Proposition}
\newtheorem{lem}[thm]{Lemma}
\newtheorem{cor}[thm]{Corollary}
\newtheorem{rmk}[thm]{Remark}

\begin{document}
\newcommand{\lip}{\textnormal{Lip}(S^1)}
\newcommand{\Lip}{\textnormal{Lip}(\sfe)}
\newcommand{\conv}{\textnormal{conv}}
\newcommand{\supp}{\textnormal{supp}}
\newcommand{\sfe}{{S^{n-1}}}
\newcommand{\R}{{\mathbb R}}
\newcommand{\N}{{\mathbb N}}
\renewcommand{\O}{{\mathbf O}}
\newcommand{\K}{{\mathcal K}}

\let\tmp\oddsidemargin
\let\oddsidemargin\evensidemargin
\let\evensidemargin\tmp
\reversemarginpar
\makeatletter
\def\cleardoublepage{\clearpage\if@twoside\ifodd\c@page\else
\hbox{}
\thispagestyle{empty}
\newpage
\if@twocolumn\hbox{}\newpage\fi\fi\fi}
\makeatother

\title{Dot product invariant valuations on $ \Lip $}
\date{}
\author{Andrea Colesanti, Daniele Pagnini, Pedro Tradacete, Ignacio Villanueva}
\maketitle
\section*{Abstract}
We provide an integral representation for continuous, rotation invariant and dot product invariant
valuations defined on the space $ \Lip $ of Lipschitz continuous functions on the unit $ n-$sphere.
\tableofcontents
\numberwithin{equation}{section}
\section{Introduction}

A \textit{valuation} on a family $ \mathcal{F} $ of sets is a function $ \nu:\mathcal{F}\rightarrow
\mathbb{R} $ such that
\begin{equation}\label{intro 1}
\nu(A\cup B)+\nu(A\cap B)=\nu(A)+\nu(B), 
\end{equation}
for every $ A,B\in\mathcal{F} $ with $ A\cup B,A\cap B\in\mathcal{F}$. The previous relation has a clear geometric meaning, being in fact a finite additivity condition. In particular, roughly speaking, every measure is a valuation.
Of particular interest (in general and for the purposes of the present paper), is the theory of valuations on the family
${\mathcal K}^n$ of {\em convex bodies} of $\R^n$ (a convex body is a compact convex subset of $\R^n$). The importance of this theory is due not only to its role in the solution of Hilbert's third problem, but also to the beauty and the profundity of some of the results that it comprises. We mention, as instances, the Hadwiger characterisation of rigid motion invariant valuations, the McMullen homogeneous decomposition for translation invariant valuations, and the Alesker irreducibility theorem, along with its consequences. For details concerning these results we refer to Chapter 6 of the monograph \cite{Schneider} by Schneider, which contains an up-to-date account on the state of the art of this theory.

The notion of valuation can be transferred to a different context, where the domain is a family of functions. Let $X$ be a space of real-valued functions; a {\em valuation} on $ X $ is a functional 
$\mu:X\rightarrow\mathbb{R}$ such that
\begin{equation}\label{intro 2}
\mu(u\vee v)+\mu(u\wedge v)=\mu(u)+\mu(v), 
\end{equation}
for every $ u,v\in X $ such that $ u\vee v,u\wedge v\in X $, where $ \vee $ and $ \wedge $ denote the pointwise maximum and pointwise minimum, respectively. The fact that \eqref{intro 2} is the natural functional counterpart of \eqref{intro 1}
can be motivated, for example, observing that if $I_A$ denotes the characteristic function of a general set $A$, then
$$
I_{A_1\cup A_2}=I_{A_1}\vee I_{A_2}\quad\mbox{and}\quad
I_{A_1\cap A_2}=I_{A_1}\wedge I_{A_2}.
$$

The study of valuations on spaces of functions started quite recently, mainly under the impulse of the rich theory of valuations on convex bodies. A branch of this area of research is focused on spaces of functions related to convexity,
such as convex functions (see \cite{Alesker-17, Cavallina/Colesanti, Colesanti-Ludwig-Mussnig-1, Colesanti-Ludwig-Mussnig-2, Colesanti-Ludwig-Mussnig-3}), log-concave functions (see \cite{Mussnig, Mussnig-2}), quasi-concave functions (see \cite{ColesantiLombardi, Colesanti-Lombardi-Parapatits}). 

A different line of investigation in the theory of valuations on function spaces traces back to the papers \cite{Klain-1, Klain-2} by Klain, concerning valuations on star-shaped sets. A subset $S$ of $\R^n$ is {\it star-shaped} (with respect to the origin) if for every $x\in S$, the segment joining $x$ and the origin is contained in $S$. To a star-shaped set $S$ we can associate its {\em radial function} $\rho_S\colon\sfe\to[0,+\infty)$ defined by:
$$
\rho_S(u)=\sup\{\lambda\ge0\colon \lambda u\in S\}.
$$
Note that for every star-shaped sets $S_1$ and $S_2$ we have
\begin{equation}\label{intro 3}
\rho_{S_1\cup S_2}=\rho_{S_1}\vee\rho_{S_2}\quad\mbox{and}\quad
\rho_{S_1\cap S_2}=\rho_{S_1}\wedge\rho_{S_2}.
\end{equation}

In \cite{Klain-1, Klain-2}, Klain considered the family $\mathcal F$ of star-shaped sets having radial functions of class $L^n(\sfe)$, endowed with the topology induced by the $L^n(\sfe)$ norm. He obtained a complete characterisation of continuous and rotation invariant valuations $\nu$ on $\mathcal F$, proving that they are all of the form
$$
\nu(S)=\int_{\sfe}F(\rho_S)dH^{n-1},\quad\forall\, S\in{\mathcal F},
$$
where $F$ is a function in $C(\R)$ verifying a suitable growth condition at infinity. 

The family $\mathcal F$ can be identified with $L^n_+(\sfe)$ (non-negative functions in $L^n(\sfe)$) and, in view of \eqref{intro 3}, to every valuation on $\mathcal F$ it corresponds a valuation on $L^n_+(\sfe)$.  Therefore the result of Klain can be rephrased as a characterisation of continuous and rotation invariant valuations on $L^n_+(\sfe)$. In this spirit, Tsang in \cite{Tsang} extended the result of \cite{Klain-1, Klain-2} to the spaces $L^p(\sfe)$, with $1\le p<\infty$ 
(see also \cite{Tradacete/Villanueva3}). 

As a natural continuation of the previous results, the third and fourth author established in \cite{Villanueva} and 
\cite{Tradacete/Villanueva1} a characterisation of continuous and rotation invariant valuations defined on star-shaped sets with continuous radial function (see also \cite{Tradacete/Villanueva2} for further extensions). The functional counterpart of their results reads as follows.

\begin{thm}[{\bf Tradacete, Villanueva}] A function $\mu\colon C(\sfe)\to\R$ is a continuous (w.r.t. uniform convergence) and rotation invariant valuation if and only if there exists $F\in C(\R)$ such that 
$$
\mu(u)=\int_{\sfe} F(u)dH^{n-1},
$$
for every $u\in C(\sfe)$.  
\end{thm}

In the present paper we consider valuations defined on a smaller space, namely the space $\Lip$ consisting of Lipschitz functions on $\sfe$ (equipped with a topology $\tau$ defined in section \ref{section Preliminaries}). The main novelty of this space is that its elements are differentiable $H^{n-1}$-a.e. on $\sfe$ (by Rademacher's theorem), and therefore we expect that derivatives will come into play. Given $u\in\Lip$, we denote by $\nabla u(x)$ the spherical gradient of $u$ at a point $x\in\sfe$, if $u$ is differentiable at $x$ (see section \ref{section Preliminaries} for definitions).
  
It can be shown that if $F\colon\R\times\R^n\times\sfe\to\R$ is continuous, then the application $\mu\colon\Lip\to\R$ defined by
$$
\mu(u)=\int_{\sfe}F(u(x),\nabla u(x),x)dH^{n-1}(x),\quad\forall\,u\in\Lip,
$$
is a continuous valuation. If we consider, as a special case, functionals of the form
\begin{equation}\label{intro 5}
\mu(u)=\int_{\sfe}F(u(x),\|\nabla u(x)\|)dH^{n-1}(x),\quad\forall\,u\in\Lip,
\end{equation}
with $F\in C(\R\times[0,\infty))$, then we have a whole family of continuous and rotation invariant valuations on $\Lip$. 
We are then led to the problem of characterising all possible continuous and rotation invariant valuations on $\Lip$.

Here we solve a special case of this problem, namely we characterise valuations $\mu$ which are additionally {\em dot product invariant}, i.e. invariant under the addition of linear functions: 
\begin{equation}\label{intro 6}
\mu(u+l)=\mu(u)
\end{equation}
for every $u\in\Lip$ and for every $l\colon\sfe\to\R$ which is the restriction to $\sfe$ of a linear function on $\R^n$.  

We can give a geometric interpretation of the previous assumption. Note that $\Lip$ contains the family ${\mathcal H}(\sfe)$ of {\em support functions} of convex bodies  (see section \ref{section Preliminaries} for definitions). If $h\in{\mathcal H}(\sfe)$ is the support function of a convex body $K$, and $l$ is the restriction to $\sfe$ of a linear function, then $h+l$ is the support function of a translated copy of $K$. Hence if $\mu$ is dot product invariant, its  restriction to support functions is ``translation invariant''. This observation is crucial for the proof of our characterisation result; indeed, dot product and rotation invariance will allow us to apply the McMullen homogeneous decomposition and the Hadwiger characterisation theorem to the restriction to ${\mathcal H}(\sfe)$ of the valuations under consideration. In particular we may say that our argument relies heavily on the theory of valuations on convex bodies. 

The main result of this paper is the following theorem.

\begin{thm}\label{characterization result}
A functional $ \mu:\Lip\rightarrow\mathbb{R} $ is a continuous, rotation invariant and dot
product invariant valuation if and only if there exist constants $ c_0,c_1,c_2\in\mathbb{R} $ such that
\begin{equation}\label{representation formula}
\mu(u)=c_0+c_1\int_{\sfe}u(x)dH^{n-1}(x)+c_2\int_{\sfe}\left[(n-1)u^2(x)-\|\nabla u(x)\|^2\right]
dH^{n-1}(x),
\end{equation}
for every $ u\in\Lip $.
\end{thm}

This result indicates that dot product invariance is a very strong condition. Indeed, the space of continuous, rotation and dot product invariant valuations has finite dimension, similarly to what happens for continuous and rigid motion invariant valuations on convex bodies, by Hadwiger's theorem. This is far from being true without dot product invariance, as shown by the class of examples \eqref{intro 5}. 

The right hand side of \eqref{representation formula} deserves some explanation. First note that, as
a function of $ u $, it is a continuous and rotation invariant valuation, being of the form \eqref{intro 5}. Concerning dot product invariance, the constant term $c_0$ is clearly invariant, the second addend is invariant as well, as the integral over the sphere of the restriction to $\sfe$ of a linear function is zero. Moreover, if $u$ is sufficiently regular, by the divergence theorem we have
$$
\int_{\sfe}\left[(n-1)u^2(x)-\|\nabla u(x)\|^2\right]
dH^{n-1}(x)=\int_{\sfe}u(x)[(n-1)u(x)+\Delta u(x)]dH^{n-1}(x),
$$
where $\Delta u$ denotes the spherical Laplacian of $u$. Dot product invariance of this term follows now from the fact that restrictions to $ \sfe $ of linear functions are eigenfunctions of the $\Delta$ operator, with eigenvalue $-(n-1)$.  

We also remark that if $u=h\in{\mathcal H}(\sfe)$ is the support function of a convex body $K$, then the right hand side of \eqref{representation formula} is just a linear combination of $V_0(K)$, $V_1(K)$ and $V_2(K)$, the first three intrinsic volumes of $K$. Therefore \eqref{representation formula} appears as a truncated version of Hadwiger's theorem; the reason why the remaining intrinsic volumes are not involved is that they do not admit an extension from ${\mathcal H}(\sfe)$ to $\Lip$. This will be clarified in the proof of Theorem \ref{characterization result} (see in particular Proposition \ref{homogeneous valuations on Lip}). 

The rest of the paper is devoted to the proof of Theorem \ref{characterization result}. However,
we will obtain in section \ref{section Approximation} an approximation result along the way (see Proposition
\ref{valuations are determined by their values on support functions}) which could be used to deal
also with valuations which are not (rotation invariant nor) dot product invariant, and, in section 
\ref{section The homogeneous decomposition}, a homogenous decomposition
for continuous and dot-product invariant valuations on $\Lip$.

\section{Preliminaries}\label{section Preliminaries}
\subsection{Basic notions}
For $ n\in\mathbb{N} $, $ n\geq 2 $, we denote by $ \sfe $ the unit $ n-$sphere, that is,
$$ \sfe=\left\{x\in\mathbb{R}^n:\|x\|=1\right\}, $$
where $ \|\cdot\| $ represents the Euclidean norm. We will use the $ (n-1)-$dimensional Hausdorff
measure $ H^{n-1} $ on the sphere. 

Even though we will mainly be interested in what happens on $ \sfe $, it will often be useful to
reason on the whole space $ \mathbb{R}^n $, where we will use the standard $ n-$dimensional
Lebesgue measure, so that every ``a.e.'' referred to functions defined on $ \mathbb{R}^n $ is to be
understood with respect to said measure. The standard basis of $ \mathbb{R}^n $ will be denoted by
$ \{e_1,\ldots,e_n\} $.

Let $ \Lip $ be the space of Lipschitz continuous maps defined on $ \sfe $, i.e., the set of
functions $ u:\sfe\rightarrow\mathbb{R} $ for which there exists a constant $ L\geq 0 $ such that
$$ \vert u(x)-u(y)\vert\leq L\|x-y\|, $$
for every $ x,y\in \sfe $. The smallest constant for which this inequality holds is called the
\textit{Lipschitz constant} associated with $ u $ and is denoted by $ L(u)$:
$$
L(u)=\sup\left\{
\frac{|u(x)-u(y)|}{\|x-y\|}\colon x,y\in \sfe,\, x\ne y
\right\}.
$$

Let $u\in\Lip$; as an application of the Rademacher theorem, $u$ is differentiable $H^{n-1}$-a.e on $\sfe$. Let $x\in\sfe$ be a point of differentiability; the differential of $u$ at $x$ is a linear application from $T_x(\sfe)$, the tangent space to $\sfe$ at $x$, to $\R$, and hence it can be represented by a vector $\nabla u(x)\in T_x(\sfe)$ that we will call the
{\em spherical gradient} of $u$ at $x$.

Let us consider on $ \Lip $ the topology $ \tau $ induced by the following convergence: we say that a sequence
$\{u_i\colon i\in\N\}\subseteq\Lip $ converges to $ u\in\Lip $ with respect to $ \tau $, in symbols $ u_i
\xrightarrow[\tau]{}u $, as $ i\rightarrow\infty $, if
\begin{enumerate}
\item $ \|u_i-u\|_\infty\rightarrow 0 $, that is, $ u_i\rightarrow u $ uniformly on $ \sfe $;
\item $ \nabla u_i(x)\rightarrow\nabla u(x) $ for $ H^{n-1}$-a.e. $ x\in\sfe $; 
\item there exists a suitable constant $ C>0 $ such that
$$ 
\|\nabla u_i(x)\|\leq C, 
$$
for every $ i\in\mathbb{N} $ and $ H^{n-1}$-a.e. $ x\in \sfe $.
\end{enumerate}

A functional $ \mu:\Lip\rightarrow\mathbb{R} $ is said to be a \textit{valuation} if
\begin{equation}\label{definition of valuation}
\mu(u\vee v)+\mu(u\wedge v)=\mu(u)+\mu(v),
\end{equation}
for every $ u,v\in\Lip $, where $ \vee $ and $ \wedge $ are the pointwise maximum and pointwise
minimum respectively. Note that, since $ \Lip $ is closed with respect to these operations, all the
terms in \eqref{definition of valuation} are well defined.

We will say that a valuation $ \mu:\Lip\rightarrow\mathbb{R} $ is \textit{continuous} if it is continuous
with respect to the topology $ \tau $, unless otherwise stated.

Besides continuity, we will be interested in other properties, namely the rotation invariance, dot
product invariance and homogeneity. These concepts are defined below.

A valuation $ \mu:\Lip\rightarrow\mathbb{R} $ is \textit{rotation invariant} if for every $ u\in\Lip $
and $ \varphi\in\O(n) $ we have
$$ \mu(u\circ\varphi)=\mu(u), $$
where $ \O(n) $ is the orthogonal group. 

$\mu$ is called \textit{dot product invariant} if, for every
$ u\in\Lip $ and $ x\in\mathbb{R}^n $,
$$ \mu(u+\langle\cdot,x\rangle)=\mu(u), $$
where $ \langle\cdot,\cdot\rangle $ denotes the standard scalar product in $\R^n$. In other words, $\mu$ is 
dot product invariant if it is invariant under the addition of linear functions restricted to $\sfe$. 

For $\alpha\ge0$, $\mu$ is \textit{$\alpha$-homogeneous} if
$$ 
\mu(\lambda u)=\lambda^\alpha\mu(u), 
$$
for every $ \lambda>0 $ and $ u\in\Lip $.

\bigskip

We will also work with valuations defined on the space $ \mathcal{K}^n $ of {\em convex bodies} of
$ \mathbb{R}^n $, namely compact and convex subsets of $\R^n$. We recall some definitions in
this context.

A \textit{valuation} on $ \mathcal{K}^n $ is a function $ \nu:\mathcal{K}^n\rightarrow\mathbb{R} $
such that
$$ \nu(K\cup L)+\nu(K\cap L)=\nu(K)+\nu(L), $$
for every $ K,L\in\mathcal{K}^n $ satisfying $ K\cup L\in\mathcal{K}^n $. Such $ \nu $ is
\textit{rotation invariant} if
$$ \nu(\varphi(K))=\nu(K), $$
for every $ K\in\mathcal{K}^n $ and $ \varphi\in\O(n) $. It is called \textit{translation
invariant} if
$$ \nu(K+x)=\nu(K), $$
for every $ K\in\mathcal{K}^n $ and $ x\in\mathbb{R}^n $. For $ 0\leq i\leq n $, $ \nu $ is
\textit{$ i-$homogeneous} if
$$ \nu(\lambda K)=\lambda^i\nu(K), $$
for every $ \lambda>0 $ and $ K\in\mathcal{K}^n $.

The \textit{Hausdorff metric} on $ \mathcal{K}^n $ is defined by
$$ d_H(K,L)=\max\left\{\sup_{x\in K}\inf_{y\in L}\|x-y\|,\;\sup_{y\in L}\inf_{x\in K}\|x-y\|\right\}, $$
for every non-empty $ K,L\in\mathcal{K}^n $.

The following set is dense in $ \mathcal{K}^n $ (see e.g. \cite{Schneider}):
$$ 
C^{2,+}=\left\{K\in\mathcal{K}^n:\partial K\in C^2\mbox{ and has strictly positive Gaussian
curvature at every point}\right\}. 
$$

Let us now recall the definition of support function: for every non-empty $ K\in\mathcal{K}^n $, its
\textit{support function} is $ h_K:\mathbb{R}^n\rightarrow\mathbb{R} $ defined by
$$ h_K(x)=\max_{y\in K}\langle x,y\rangle,\;x\in\mathbb{R}^n. $$
Support functions are convex and 1-homogeneous, that is, $ h_K(\lambda x)=\lambda h_K(x) $ for
every $ \lambda>0 $ and $ x\in\mathbb{R}^n $. Moreover, $ \|h_K-h_L\|_\infty=d_H(K,L) $,
for every non-empty $ K,L\in\mathcal{K}^n $.

\bigskip

The notion of piecewise linear function will also be useful. A continuous function $ f:\mathbb{R}^n
\rightarrow\mathbb{R} $ is said to be a \textit{piecewise linear function} if there exist closed convex
cones $ C_1,\ldots,C_m $ with vertex at the origin and pairwise disjoint interiors such that
$$ 
\bigcup_{i=1}^m C_i=\mathbb{R}^n, 
$$
and linear functions $L_i\colon\R^n\to\R$, $i=1,\dots, m$, such that $ f=L_i $ on $ C_i $, for $ i=1,\ldots,m $.

We denote by $ \mathcal{H}(\sfe) $ and $ \mathcal{L}(\sfe) $ the sets of the restrictions to
$ \sfe $ of support functions and piecewise linear functions respectively. When considering these
same functions defined on the whole space $ \mathbb{R}^n $ we will use the symbols
$ \mathcal{H}(\mathbb{R}^n) $ and $ \mathcal{L}(\mathbb{R}^n) $ instead.
Note that, since support functions are convex, they are locally Lipschitz continuous, hence
$ \mathcal{H}(\sfe)\subseteq\Lip $. We also have the inclusion $ \mathcal{L}(\sfe)\subseteq\Lip$.

We introduce one last notation, which will come in handy in the future, denoting by
$$ \widehat{\mathcal{H}}(\sfe)=\left\{\bigwedge_{i=1}^m h_{K_i}:m\in\mathbb{N},\;
h_{K_i}\in\mathcal{H}(\sfe)\mbox{ for }i=1,\ldots,m\right\} $$
the space of finite minima of support functions.

\subsection{The theorems of Hadwiger and McMullen for valuations on convex bodies}
In section \ref{section The homogeneous decomposition} we will prove a McMullen-type
decomposition result for continuous and dot product invariant valuations; we hereby recall the
original McMullen decomposition theorem for valuations defined on $\mathcal{K}^n $, which will be
used in the proof of our result in section \ref{section The homogeneous decomposition}.

\begin{thm}[McMullen, \cite{Schneider}]\label{McMullen}
Let $ \nu:\mathcal{K}^n\rightarrow\mathbb{R} $ be a translation invariant valuation which is 
continuous with respect to the Hausdorff metric. Then there exist continuous and translation invariant
valuations $ \nu_0,\ldots,\nu_n:\mathcal{K}^n\rightarrow\mathbb{R} $ such that $ \nu_i $ is
$ i-$homogeneous, for $ i=0,\ldots,n $, and
\begin{equation}\label{McMullen decomposition formula}
\nu(\lambda K)=\sum_{i=0}^n\lambda^i\nu_i(K),
\end{equation}
for every $ K\in\mathcal{K}^n $ and $ \lambda>0 $.
\end{thm}

The proof of Theorem \ref{characterization result} is based upon the famous Hadwiger theorem,
recalled below.

\begin{thm}[Hadwiger, \cite{Schneider}]\label{Hadwiger}
A map $ \nu:\mathcal{K}^n\rightarrow\mathbb{R} $ is a rotation and translation invariant valuation
which is continuous with respect to the Hausdorff metric if and only if there exist constants
$ c_0,\ldots,c_n\in\mathbb{R} $ such that
\begin{equation}\label{Hadwiger's formula}
\nu(K)=\sum_{i=0}^nc_iV_i(K),
\end{equation}
for every $ K\in\mathcal{K}^n $, where $ V_i $ denotes the $i^{\textnormal{th}}$ intrinsic volume.
\end{thm}

For the definition and properties of the intrinsic volumes, see \cite{Schneider}.

\subsection{McShane's lemma}

Let $\varphi:\mathbb{R}^n\rightarrow\mathbb{R}$, and assume that it is differentiable at a point $x\in\R^n$; we denote
by $\nabla_e\varphi(x)$ the standard Euclidean gradient of $\varphi$ at $x$. 

Given a function $u\colon\sfe\to\R$, it can often be convenient to extend it to $\R^n$ as a 1-homogeneous function. 
Let us denote by $\tilde u$ this extension:
$$
\tilde u(x)=
\left\{
\begin{array}{lll}
\|x\|\, u\left(\displaystyle{\frac x{\|x\|}}\right)&\quad\mbox{if $x\ne0$},\\
\\
0&\quad\mbox{if $x=0$.}
\end{array}
\right.
$$
Let $x\in\sfe$ be a point where $u$ is differentiable; then $\tilde u$ is differentiable at $x$ as well. If we denote by
$\nabla_e \tilde u(x)$ the standard Euclidean gradient of $\tilde u$ at $x$, by the Euler relation we obtain the following equality:
\begin{equation}\label{connection between the gradients}
\|\nabla_e\tilde u(x)\|^2=\|\nabla u(x)\|^2+u^2(x).
\end{equation}

Let $u\in\Lip$; formula \eqref{connection between the gradients} can be applied to get the following bound on 
the spherical gradient:
\begin{equation}\label{bound on the gradient of a lipschitz function}
\|\nabla u(x)\|\leq\sqrt{n}\cdot L(u),
\end{equation}
for $ H^{n-1}$-a.e. $ x\in \sfe $.


There is another way of extending Lipschitz functions defined on $\sfe$ which we will be interested in; it is stated in the
following theorem.
\begin{thm}[McShane, \cite{McShane}]\label{McShane}
Let $ S\subseteq\mathbb{R}^n $ and $ u:S\rightarrow\mathbb{R} $ be a Lipschitz function with
Lipschitz constant $ L $. Then the map $ \bar{u}:\mathbb{R}^n\rightarrow\mathbb{R} $ defined by
$$ 
\bar{u}(x)=\sup_{z\in S}[u(z)-L\|x-z\|], 
$$
for $ x\in\mathbb{R}^n $, is still Lipschitz continuous with the same Lipschitz constant $ L $.
\end{thm}
We conclude this paragraph by recalling that the gradient of a support function possesses the
following property (see \cite{Schneider}).
\begin{prop}\label{gradient of support functions}
Let $ K\in\mathcal{K}^n $. If $ h_K $ is differentiable at $ x\in \sfe $, then $ \nabla_eh_K(x)\in
\partial K $, and $ \nabla_eh_K(x) $ is the only point of $ \partial K $ with outer normal vector $ x $.
\end{prop}

\subsection{Some remarks}
In this paragraph we collect some technical remarks which will be useful throughout the paper.
It is convenient to study the behaviour of support functions with respect to the operators
$ \vee $ and $ \wedge $. The result hereby presented is well-known, we include the proof for completeness.

\begin{lem}\label{maximum and minimum of support functions}
Let $ K,L\in\mathcal{K}^n $. Then $ h_K\vee h_L=h_{\conv(K\cup L)} $, where $ \conv(K\cup L) $
denotes the convex hull of $ K\cup L $. Moreover, if $ K\cup L\in\mathcal{K}^n $ we have
$$ 
h_K\vee h_L=h_{K\cup L},\quad h_K\wedge h_L=h_{K\cap L}. 
$$
\end{lem}

\begin{proof}
Clearly $ \conv(K\cup L) $ contains $ K $ and $ L $, thus $ h_{\conv(K\cup L)}\geq h_K $ 
and $ h_{\conv(K\cup L)}\geq h_L $. This implies the inequality
$$ 
h_{\conv(K\cup L)}(x)\geq h_K\vee h_L(x), 
$$
for every $ x\in\mathbb{R}^n $. Vice versa, if $ x\in\mathbb{R}^n $, then
$$ h_{\conv(K\cup L)}(x)=\max_{z\in\conv(K\cup L)}\langle x,z\rangle, $$
where the maximum will be attained in correspondence of a certain element $ z=\sum_{i=1}^m
\lambda_ix_i $, with $ m\in\mathbb{N} $, $ \lambda_i\geq 0 $, $ \sum_{i=1}^m\lambda_i=1 $,
$ x_i\in K\cup L $. By reordering the elements, we can assume that $ 
\{x_1,\ldots,x_l\}\subseteq K $
and $ \{x_{l+1},\ldots,x_m\}\subseteq L $. Therefore, we have
\begin{equation}\notag
\begin{split}
h_{\conv(K\cup L)}(x)&=\left\langle x,\sum_{i=1}^m\lambda_ix_i\right\rangle=
\sum_{i=1}^l\lambda_i\langle x,x_i\rangle+\sum_{i=l+1}^m\lambda_i\langle 
x,x_i\rangle\\
&\leq\sum_{i=1}^l\lambda_ih_K(x)+\sum_{i=l+1}^m\lambda_ih_L(x)\\
&\leq\left(\sum_{i=1}^l\lambda_i+\sum_{i=l+1}^m\lambda_i\right)h_K\vee h_L(x)=h_K\vee 
h_L(x).
\end{split}
\end{equation}

We now work under the hypothesis $ K\cup L\in\mathcal{K}^n $. We first prove that
\begin{equation}\label{Minkowski sum of union and intersection}
(K\cup L)+(K\cap L)=K+L.
\end{equation}

Note that
$$ (K\cup L)+(K\cap L)\subseteq K+L. $$
Vice versa, let $ x+y\in K+L $. If either $ x\in K\cap L $ or $ y\in K\cap L $, then we 
are done;
suppose now $ x\in K\backslash L $ and $ y\in L\backslash K $. Because of this 
assumption, there
exists $ t\in(0,1) $ such that
$$ z:=tx+(1-t)y\in K\cap L, $$
since $ K\cup L\in\mathcal{K}^n $. Therefore,
$$ x+y=(1-t)x+ty+z\in(K\cup L)+(K\cap L), $$
using the convexity of $ K\cup L $ again. Formula \eqref{Minkowski sum of union and 
intersection}
follows.
From the properties of support functions (see \cite[Section 1.7]{Schneider}) we obtain
\begin{eqnarray*}
h_{K\cup L}+h_{K\cap L}&=&h_{(K\cup L)+(K\cap L)}=h_{K+L}=h_K+h_L\\
&=&h_K\vee h_L+h_K\wedge h_L=h_{\conv(K\cup L)}+h_K\wedge h_L=h_{K\cup L}+h_K\wedge h_L, 
\end{eqnarray*}
where the last equality follows from the hypothesis $ K\cup L\in\mathcal{K}^n $. This 
proves the lemma.

\end{proof}

We are now going to state a topological result concerning the continuity of a functional $ \mu:
\mathcal{H}(\sfe)\rightarrow\mathbb{R} $. By definition of $ \tau $, we have that if such a
functional is continuous with respect to $ \|\cdot\|_\infty $, then it is also continuous with respect to
$ \tau $. The converse is also true.

\begin{lem}\label{continuity of a valuation on H}
Let $ \mu:\mathcal{H}(\sfe)\rightarrow\mathbb{R} $. Then $ \mu $ is continuous with respect to
$ \tau $ if and only if it is continuous with respect to $ \|\cdot\|_\infty $.
\end{lem}

\begin{proof}
Consider a $ \tau-$continuous functional $ \mu:\mathcal{H}(\sfe)\rightarrow\mathbb{R} $
and a sequence $ \{h_{K_i}\}\subseteq\mathcal{H}(\sfe) $ of support functions such that $ \|h_{K_i}-
h_K\|_\infty\rightarrow 0 $, as $ i\rightarrow\infty $, where $K\in{\mathcal K}^n$. It is enough to prove that $ h_{K_i}
\xrightarrow[\tau]{}h_K $.

Define
$$ D_i=\{x\in \sfe:h_{K_i}\mbox{ is differentiable at }x\}, $$
for $ i\in\mathbb{N} $, and
$$ D_0=\{x\in \sfe:h_K\mbox{ is differentiable at }x\}. $$
We also set
$$ D=\bigcap_{i=0}^\infty D_i. $$
Note that $ H^{n-1}(\sfe\backslash D)=0 $, because of Rademacher's theorem.

For every $ x\in D $ we have
\begin{equation}\label{convergence of gradients}
\nabla h_{K_i}(x)\rightarrow\nabla h_K(x).
\end{equation}
Indeed, consider a subsequence $ \{h_{K_{i_j}}\}\subseteq\{h_{K_i}\} $. For every $ j\in\mathbb{N} $,
the differentiability of $ h_{K_{i_j}} $ at $ x $ implies (see, e.g., \cite[Section 1.5]{Schneider})
\begin{equation}\label{subgradient inequality}
h_{K_{i_j}}(y)\geq h_{K_{i_j}}(x)+\langle\nabla_eh_{K_{i_j}}(x),y-x\rangle,
\end{equation}
for every $ y\in\mathbb{R}^n $. The condition $ \|h_{K_{i_j}}-h_K\|_\infty\rightarrow 0 $ implies
that $ K_{i_j}\rightarrow K $ with respect to the Hausdorff metric, hence there is a convex body
$ \widetilde{K} $ such that $ K_{i_j}\subseteq\widetilde{K} $, for every $ j\in\mathbb{N} $ (see
\cite[Section 1.8]{Schneider}). From Proposition \ref{gradient of support functions} we have that
$$ \nabla_eh_{K_{i_j}}(x)\in\partial K_{i_j}\subseteq K_{i_j}\subseteq\widetilde{K}, $$
thus there is a subsequence $ \{h_{K_{i_{j_l}}}\}\subseteq\{h_{K_{i_j}}\} $ such that $ \lim_{l
\rightarrow\infty}\nabla_eh_{K_{i_{j_l}}}(x) $ exists, by the Bolzano-Weierstrass theorem. Writing
\eqref{subgradient inequality} for this subsequence and letting $ l\rightarrow\infty $ we obtain
$$ h_K(y)\geq h_K(x)+\left\langle\lim_{l\rightarrow\infty}\nabla_eh_{K_{i_{j_l}}}(x),y-x\right\rangle, $$
for every $ y\in\mathbb{R}^n $. Recalling the uniqueness of the subgradient at differentiability points
for convex functions (\cite[Section 1.5]{Schneider}), the last inequality implies
$$ \lim_{l\rightarrow\infty}\nabla_eh_{K_{i_{j_l}}}(x)=\nabla_eh_K(x). $$
This, together with relation \eqref{connection between the gradients} and the arbitrariness of
$ \{h_{K_{i_j}}\}\subseteq\{h_{K_i}\} $, proves \eqref{convergence of gradients}.

Moreover, for every $ x\in D $ and $ i\in\mathbb{N} $ we have
$$
\|\nabla h_{K_i}(x)\|\le\|\nabla_e h_{K_i}(x)\|\le \max\{\|y\|\colon y\in\widetilde K\},
$$
where we have used Proposition \ref{gradient of support functions} again. Hence we have a uniform
bound on $ \|\nabla h_{K_i}\| $ in $ D $. Thus $ h_{K_i}\xrightarrow[\tau]{}h_K $, as desired.
\end{proof}

Theorems \ref{McMullen} and \ref{Hadwiger} concern valuations on convex bodies, and since we will
be interested in studying valuations on support functions using these results, it would be nice to know
that we can ``move'' valuations from $ \mathcal{H}(\sfe) $ to $ \mathcal{K}^n $ without losing any
property. This is stated precisely in the next result.
\begin{lem}\label{moving valuations}
Let $ \mu:\mathcal{H}(\sfe)\rightarrow\mathbb{R} $. Define $ \nu:\mathcal{K}^n\rightarrow
\mathbb{R} $ by setting
$$ \nu(K)=\mu(h_K), $$
for every $ K\in\mathcal{K}^n $. Then
\begin{itemize}
\item[i)] if $ \mu $ is a valuation, then so is $\nu$;
\item[ii)] if $ \mu $ is $ \tau-$continuous, then $\nu$ is continuous with respect to the Hausdorff metric;
\item[iii)] if $ \mu $ is rotation invariant, then so is $\nu$;
\item[iv)] if $ \mu $ is dot product invariant, then $\nu$ is translation invariant;
\item[v)] if $ \mu $ is $ i-$homogeneous for some $ i\in\{0,\ldots,n\} $, then so is $ \nu $.
\end{itemize}
\end{lem}
Assertions $ i) $ and $ ii) $ are consequences of lemmas
\ref{maximum and minimum of support functions} and \ref{continuity of a valuation on H} respectively.
The rest of the statement follows from the properties of support functions, which can be found in
\cite{Schneider}.

\section{Approximation}\label{section Approximation}
The idea behind the proof of Theorem \ref{characterization result} is that of using an approximation
result to narrow down the study of our valuations from the space $ \Lip $ to its subset
$ \mathcal{H}(\sfe) $, which is in bijection with $ \mathcal{K}^n $, where Hadwiger's theorem can
be applied.

More precisely, our goal is to prove that continuous valuations on $ \Lip $ are uniquely determined by
the values they attain at support functions, as stated in the following proposition.

\begin{prop}\label{valuations are determined by their values on support functions}
Let $ \mu_1,\mu_2:\Lip\rightarrow\mathbb{R} $ be continuous valuations. If $ \mu_1=\mu_2 $ on
$ \mathcal{H}(\sfe) $, then $ \mu_1=\mu_2 $ on $ \Lip $.
\end{prop}
The proof is split into four main steps, which will be detailed in the next paragraphs.
\subsection{$ \mathcal{L}(\sfe)\subseteq\widehat{\mathcal{H}}(\sfe) $}
First of all, we are going to prove that piecewise linear functions can be written as finite minima of
support functions.
\begin{lem}\label{piecewise linear functions are minima of support functions}
Let $ f\in\mathcal{L}(\mathbb{R}^n) $. Then there exist $ m\in\mathbb{N} $ and $ h_{K_1},
\ldots,h_{K_m}\in\mathcal{H}(\mathbb{R}^n) $ such that
$$ 
f=\bigwedge_{i=1}^mh_{K_i}. 
$$
In particular, $ \mathcal{L}(\sfe)\subseteq\widehat{\mathcal{H}}(\sfe) $.
\end{lem}

\begin{proof}
For $ f\in\mathcal{L}(\mathbb{R}^n) $, there are closed convex cones $ C_1,\ldots,C_m $ with vertex at 
the origin and pairwise disjoint interiors such that
$$ \bigcup_{i=1}^mC_i=\mathbb{R}^n, $$
and $ f=L_i $ is linear on $ C_i $, for $ i=1,\ldots,m $.

We will now focus on the cone $ C_1 $. Consider $ \tilde{f}=f-L_1 $. Let 
$P_{C_1}\colon\R^n\to\R^n$ denote the
so called metric projection onto $C_1$: for every $x\in\R^n$, $P_{C_1}(x)$ is the unique point in $C_1$ such that
$$
\|x-P_{C_1}(x)\|=\min_{z\in C_1}\|x-z\|.
$$
As $C_1$ is closed and convex, this function is well defined.
We also define the function $ g:\mathbb{R}^n\rightarrow\mathbb{R} $ by
$$ 
g(x)=\|x-P_{C_1}(x)\|=\min_{z\in C_1}\|x-z\|, 
$$
for every $ x\in\mathbb{R}^n $; this is the distance function from the cone $ C_1 $. As $C_1$ is a convex cone with apex at the origin, $ g $ is 1-homogeneous and subadditive, hence it is also convex. These properties imply the existence of a
convex body $ K\in\mathcal{K}^n $ such that $ g=h_K $ (see \cite[Section 1.8]{Schneider}).

We prove that there exists a suitable constant $ c>0 $ such that
\begin{equation}\label{existence of the constant c}
cg(x)\geq\tilde{f}(x),
\end{equation}
for every $ x\in\mathbb{R}^n $. Suppose this to be false; then for every $ c>0 $ there exists a point
$ x_c\in\mathbb{R}^n $ such that $ cg(x_c)<\tilde{f}(x_c) $. Choosing $ c=i $, $ i\in\mathbb{N} $,
we construct a sequence $ \{x_i\}\subseteq\mathbb{R}^n $ satisfying
\begin{equation}\label{inequality for x_i}
g(x_i)<\frac{1}{i}\tilde{f}(x_i),
\end{equation}
for every $ i\in\mathbb{N} $. Because this inequality is strict, we have that $ x_i\neq 0 $ for every
$ i\in\mathbb{N} $. From the 1-homogeneity we get
$$ g\left(\frac{x_i}{\|x_i\|}\right)<\frac{1}{i}\tilde{f}\left(\frac{x_i}{\|x_i\|}\right). $$
This means that in \eqref{inequality for x_i} we may assume that
$ \{x_i\}\subseteq \sfe $ and, up to passing to a subsequence, $ x_i\rightarrow x $ as $ i\rightarrow\infty $, for some $ x\in\sfe $.

We observe that $ x\in C_1 $. In fact, if $ x\in\mathbb{R}^n\backslash C_1 $, then letting
$ i\rightarrow\infty $ in \eqref{inequality for x_i} we would have $ 0<g(x)\leq 0 $, a contradiction.

Let $ \tilde{x}_i=P_{C_1}(x_i) $; by continuity of the projection, $ \tilde{x}_i\rightarrow x $ as
$ i\rightarrow\infty $. From \eqref{inequality for x_i}, using the fact that $ \tilde{f}(\tilde{x}_i)=0 $
(since $ \tilde{x}_i\in C_1 $) and setting $ \widetilde{L}=L(\tilde{f}) $, we get
$$ \|x_i-\tilde{x}_i\|=g(x_i)<\frac{1}{i}\left[\tilde{f}(x_i)-\tilde{f}(\tilde{x}_i)\right]\leq\frac{1}{i}
\widetilde{L}\|x_i-\tilde{x}_i\|, $$
for every $ i\in\mathbb{N} $. Since $ x_i\not\in C_1 $ (because of \eqref{inequality for x_i} and the
non-negativity of $ g $), whereas $ \tilde{x}_i\in C_1 $, we have $ x_i\neq\tilde{x}_i $, and so the last
inequality yields $ i<\widetilde{L} $, for every $ i\in\mathbb{N} $; letting $ i\rightarrow\infty $ we
obtain a contradiction. Then there must be a constant $ c>0 $ such that
\eqref{existence of the constant c} holds for every $ x\in\mathbb{R}^n $.

Therefore,
$$ f_1:=cg+L_1\geq\tilde{f}+L_1=f $$
on $ \mathbb{R}^n $, with $ f_1=L_1=f $ on the cone $ C_1 $. Furthermore, if $ L_1(x)=\langle x,a_1
\rangle $, we have that
$$ f_1=cg+L_1=ch_K+h_{\{a_1\}}=h_{cK+a_1}=:h_{K_1} $$
is a support function.

We repeat the process for each cone $ C_i $, $ i=2,\ldots,m $, building support functions
$ h_{K_i}:\mathbb{R}^n\rightarrow\mathbb{R} $ such that $ h_{K_i}\geq f $ on $ \mathbb{R}^n $ and
$ h_{K_i}=f $ on $ C_i $. Thus we can write
$$ f=\bigwedge_{i=1}^m h_{K_i}. $$
\end{proof}

\subsection{Approximation of $ C^1 $ functions by piecewise linear functions}
Piecewise linear functions can be used to approximate $ C^1 $ functions with respect to the topology
$ \tau $, as stated in the following lemma.
\begin{lem}\label{approximation of C^1 functions}
Let $ u\in C^1(\sfe) $. Then there exists a sequence $ \{f_i\}\subseteq\mathcal{L}(\sfe) $ such
that $ f_i\xrightarrow[\tau]{}u $.
\end{lem}
To prove this we will need a preliminary observation, which in turn requires a definition: a partition
$ \mathcal{P} $ of a set $ Q\subseteq\mathbb{R}^n $ is called a \textit{simplicial partition} if it is
made up of simplices such that for every two of them, their intersection is either empty or a face (of any dimension
between $0$ and $n-1$) of both simplices. We can now point out the following fact.
\begin{rmk}\label{simplicial partition}
Let $ Q\subseteq\mathbb{R}^n $ be an $ n-$dimensional hypercube. Then there exists a simplicial 
partition $ \mathcal{P} $ of $ Q $ in $n$-dimensional simplices which induces isometric simplicial partitions in $(n-1)$-dimensional simplices on the facets of $Q$.
\end{rmk}
\begin{proof}
We proceed by induction on the dimension $ n $. If $ n=2 $, then $ Q $ is a square, and we can choose
$ \mathcal{P} $ to be the partition made up of the four simplices obtained by connecting the center of
the square to each of the four vertices.

Let $ n\geq 3 $. Since the facets of an $ n-$cube are $ (n-1)-$cubes, we can apply the inductive 
hypothesis to a facet $ F $ of $ Q $ to obtain a simplicial partition of it. We replicate this same partition 
on each facet of $ Q $. This gives a simplicial partition of the boundary $ \partial Q $ of $ Q $. We now connect the center of $Q$ to each vertex of every simplex in the aforementioned partition of $ \partial Q $, through a segment. This gives a simplicial partition of $ Q $ with the required property.
\end{proof}
This allows us to prove the lemma stated above.
\begin{proof}[Proof of Lemma \ref{approximation of C^1 functions}]
Let $ u\in C^1(\sfe) $ and consider its 1-homogeneous extension $ \tilde{u} $ to $ \mathbb{R}^n$:
$$ 
\tilde{u}(x):=\|x\|\cdot u\left(\frac{x}{\|x\|}\right),\;x\in\mathbb{R}^n\backslash\{0\},\quad \tilde u(0)=0. 
$$
Then $ \tilde{u}\in C^1(\mathbb{R}^n\backslash\{0\}) $. In particular, $ \tilde{u}\in C^1(D) $, where
$$ 
D=\left\{x\in\mathbb{R}^n\;\middle|\;1\leq\|x\|\leq\sqrt{n}\right\}. 
$$

Fix $ \varepsilon>0 $. Since $ \tilde{u} $ and its Euclidean gradient $ \nabla_e\tilde{u} $ are uniformly
continuous on $ D $, there exists $ \delta>0 $ such that for every $ x,y\in D $ with $ \|x-y\|\leq\delta $
we have
\begin{equation}\label{uniform continuity of u tilde}
\vert\tilde{u}(x)-\tilde{u}(y)\vert<\varepsilon
\end{equation}
and
\begin{equation}\label{uniform continuity of the gradient of u tilde}
\|\nabla_e\tilde{u}(x)-\nabla_e\tilde{u}(y)\|<\varepsilon.
\end{equation}

Consider now the hypercube $ \Omega=[-1,1]^n $ centered at the origin with edge of length $ 2 $.
For each coordinate axis, we draw hyperplanes orthogonal to such axis so that $ \Omega $ is cut into
hypercubes with edges of length $ \frac{1}{N} $, where $ N=\left\lceil\frac{\sqrt{n}}{\delta}\right\rceil $
($ \lceil\cdot\rceil $ denotes the ceiling function). Note that these hypercubes all have the same
diameter $ d\leq\delta $.

We apply Remark \ref{simplicial partition} to the facets of these hypercubes (which are
$ (n-1)-$dimensional hypercubes) that are contained in $ \partial\Omega $: this determines a
simplicial partition $ \{\Delta_1,\ldots,\Delta_m\} $ in $ (n-1)-$simplices of $ \partial\Omega $.
For every $i=1,\dots,m$, let
$$
C_i=\{tx\colon t\ge0,\, x\in\Delta_i\}.
$$
Then $C_1,\dots,C_m$ are closed convex cones with pairwise disjoint interiors, which form a partition 
of the whole space $ \mathbb{R}^n $. Note that since the annulus $ D $ contains all the simplices $ \Delta_i $ and
$ d\leq\delta $, formulas \eqref{uniform continuity of u tilde} and
\eqref{uniform continuity of the gradient of u tilde} are satisfied for every $ x $ and $ y $ belonging
to the same simplex.

We consider linear maps $ L_i:C_i\rightarrow\mathbb{R} $, $ i=1,\ldots,m $, such that $ L_i $ 
coincides with $ \tilde{u} $ on each of the $ n $ vertices of $ \Delta_i $; these maps are uniquely
determined. Let $ f\in\mathcal{L}(\mathbb{R}^n) $ be the continuous function such that $ f=L_i $ on
$ C_i $, for $ i=1,\ldots,m $, and define $ \psi=\tilde{u}-f $.

For a fixed $ x\in D $, we have that $ x\in C_k $ for some $ k\in\{1,\ldots,m\} $, and we can write
$ x=\lambda x' $, with $ \lambda>0 $ and $ x'\in\Delta_k $. Choose an arbitrary vertex $ v $ of
$ \Delta_k $. Since $ \Delta_k $ is compact, there is a $ w\in\Delta_k $ such that
$$ \vert L_k(v)-L_k(w)\vert=\max_{z\in\Delta_k}\vert L_k(v)-L_k(z)\vert. $$
Since convex functions on convex polytopes attain their maximum at the vertices, $ w $ must be a
vertex of $ \Delta_k $.

Given that both $ \tilde{u} $ and $ f $ are 1-homogeneous, using the triangular inequality,
\eqref{uniform continuity of u tilde} and the fact that $ \tilde{u}=L_k $ on the vertices of $ \Delta_k $,
we get
\begin{eqnarray*}
\vert\psi(x)\vert&=&\lambda\vert\tilde{u}(x')-f(x')\vert=\lambda\vert\tilde{u}(x')-L_k(x')\vert<
\lambda\varepsilon+\lambda\vert L_k(v)-L_k(x')\vert\\
&\leq&\lambda\varepsilon+\lambda\vert L_k(v)-L_k(w)\vert=\lambda\varepsilon+\lambda\vert
\tilde{u}(v)-\tilde{u}(w)\vert<2\lambda\varepsilon=2\frac{\|x\|}{\|x'\|}\varepsilon\\
&\leq& 2\sqrt{n}
\varepsilon.
\end{eqnarray*}
Therefore,
\begin{equation}\label{uniform approximation}
\|\psi\|_{\infty,D}<2\sqrt{n}\varepsilon,
\end{equation}
where $ \|\cdot\|_{\infty,D} $ denotes the uniform norm on $ D $.

We now turn to the gradient $ \nabla_e\psi $. Fix arbitrarily $ k\in\{1,\ldots,m\} $ and $ x\in
\textnormal{relint}\left(\Delta_k\right) $, the relative interior of $ \Delta_k $ (i.e. the interior of 
$\Delta_k$ as a subset of an $(n-1)$-dimensional affine hyperplane of $\R^n$). Then $ \nabla_e\psi(x) $
exists. We choose a vertex of $ \Delta_k $ and consider the $ n-1 $ edges $ l_1,\ldots,l_{n-1} $
incident to it; since $ \Delta_k $ is a simplex, these edges lie on linearly independent directions
$ \nu_1,\ldots,\nu_{n-1} $. Note that the restriction of $ \psi $ to each $ l_i $ can be seen as a
function of one variable which is continuous on $ l_i $, differentiable in its relative interior and satisfies
$ \psi=0 $ at the ends of $ l_i $. Hence there exist points $ z_i\in l_i $ such that
$$ 
\frac{\partial\psi}{\partial\nu_i}(z_i)=0, 
$$
for every $ i=1,\ldots,n-1 $. Using the fact that $ f|_{C_k} $ is linear and the Cauchy-Schwarz
inequality, for every $ i=1,\ldots,n-1 $ we get
\begin{eqnarray}\label{small derivatives in the first system}
\Bigg|\frac{\partial\psi}{\partial\nu_i}(x)\Bigg|&=&\Bigg|\frac{\partial\psi}{\partial\nu_i}(x)-
\frac{\partial\psi}{\partial\nu_i}(z_i)\Bigg|=\Bigg|\frac{\partial\tilde{u}}{\partial\nu_i}(x)-f(\nu_i)-
\frac{\partial\tilde{u}}{\partial\nu_i}(z_i)+f(\nu_i)\Bigg|\nonumber\\
&=&\vert\langle\nabla_e\tilde{u}(x)-\nabla_e\tilde{u}(z_i),\nu_i\rangle\vert\leq
\|\nabla_e\tilde{u}(x)-\nabla_e\tilde{u}(z_i)\|<\varepsilon,
\end{eqnarray}
where the last inequality follows from \eqref{uniform continuity of the gradient of u tilde}.

Let $ H $ be the hyperplane passing through the $ n $ vertices of $ \Delta_k $. Since $ \Delta_k
\subseteq\partial\Omega $, the exterior unit normal vector $ N $ to $ H $ is of the form $ N=\pm
e_{i_k} $, for some $ i_k\in\{1,\ldots,n\} $; for the sake of simplicity, let us assume $ N=e_n $, since
the general case can be dealt with analogously. We now observe that both $ \{\nu_1,\ldots,
\nu_{n-1}\} $ and $ \{e_1,\ldots,e_{n-1}\} $ are bases of $ H $; in particular, there exist numbers
$ \alpha_{ij}\in\mathbb{R} $, $ i,j=1,\ldots,n-1 $, such that
$$ e_i=\sum_{j=1}^{n-1}\alpha_{ij}\nu_j, $$
for every $ i=1,\ldots,n-1 $. Therefore, for $ i=1,\ldots,n-1 $ we have
\begin{equation}\label{small derivatives in the second system}
\Bigg|\frac{\partial\psi}{\partial x_i}(x)\Bigg|=\vert\langle\nabla_e\psi(x),e_i\rangle\vert\leq
\sum_{j=1}^{n-1}\vert\alpha_{ij}\vert\cdot\Bigg|\frac{\partial\psi}{\partial\nu_j}(x)\Bigg|
<M\varepsilon,
\end{equation}
where we have used \eqref{small derivatives in the first system} and we have set
$$ M=\max_{i\in\{1,\ldots,n-1\}}\sum_{j=1}^{n-1}\vert\alpha_{ij}\vert. $$
Note that the $ \alpha_{ij}$'s, and thus $ M $, do not depend on $ \varepsilon $, since they are 
determined by the $ \nu_i$'s, which are in turn determined by the simplicial partitions of the
$ (n-1)-$cubes contained in $ \partial\Omega $; the length $ \frac{1}{N} $ of the edge of such cubes
depends on $ \varepsilon $, but the angles appearing in the aformentioned partitions,
hence the $ \nu_i$'s, do not.

We now write
$$ \nabla_e\psi(x)=\langle\nabla_e\psi(x),e_1\rangle e_1+\ldots+\langle\nabla_e\psi(x),e_{n-1}
\rangle e_{n-1}+\langle\nabla_e\psi(x),e_n\rangle e_n, $$
which, together with \eqref{small derivatives in the second system}, implies
\begin{equation}\label{bound for the gradient of psi}
\|\nabla_e\psi(x)\|<M(n-1)\varepsilon+\Bigg|\frac{\partial\psi}{\partial x_n}(x)\Bigg|.
\end{equation}

Let us consider the radial direction $ r_x=\frac{x}{\|x\|} $. We have
$$ \frac{\partial\psi}{\partial r_x}(x)=\langle\nabla_e\psi(x),r_x\rangle=\frac{\psi(x)}{\|x\|}, $$
thanks to Euler's formula for homogeneous functions. This yields
\begin{equation}\label{small derivative in the radial direction}
\Bigg|\frac{\partial\psi}{\partial r_x}(x)\Bigg|=\frac{\vert\psi(x)\vert}{\|x\|}\leq\vert\psi(x)\vert
<2\sqrt{n}\varepsilon,
\end{equation}
because of \eqref{uniform approximation}. On the other hand, $ r_x=\langle r_x,e_1\rangle e_1
+\ldots+\langle r_x,e_n\rangle e_n $, hence
$$ \frac{\partial\psi}{\partial r_x}(x)=\langle\nabla_e\psi(x),r_x\rangle=\sum_{i=1}^{n-1}\left[
\langle r_x,e_i\rangle\cdot\frac{\partial\psi}{\partial x_i}(x)\right]+\langle r_x,e_n\rangle\cdot
\frac{\partial\psi}{\partial x_n}(x), $$
so that
\begin{eqnarray*}
\vert\langle r_x,e_n\rangle\vert\cdot\Bigg|\frac{\partial\psi}{\partial x_n}(x)\Bigg|&\leq&\Bigg|
\frac{\partial\psi}{\partial r_x}(x)\Bigg|+\sum_{i=1}^{n-1}\vert\langle r_x,e_i\rangle\vert
\cdot\Bigg|\frac{\partial\psi}{\partial x_i}(x)\Bigg|\\
&<&2\sqrt{n}\varepsilon+\sum_{i=1}^{n-1}M\varepsilon=[2\sqrt{n}+M(n-1)]\varepsilon, 
\end{eqnarray*}
where we have used \eqref{small derivative in the radial direction},
\eqref{small derivatives in the second system} and the Cauchy-Schwarz inequality. But since
$$ \vert\langle r_x,e_n\rangle\vert=\vert\langle r_x,N\rangle\vert=\frac{1}{\|x\|}\geq
\frac{1}{\sqrt{n}}, 
$$
we obtain
$$ \Bigg|\frac{\partial\psi}{\partial x_n}(x)\Bigg|<\left[2n+M\sqrt{n}(n-1)\right]\varepsilon. $$

From \eqref{bound for the gradient of psi} we conclude that
\begin{equation}\label{small gradient on the simplices}
\|\nabla_e\psi(x)\|<C\varepsilon,
\end{equation}
where
$$ C=2n+M(\sqrt{n}+1)(n-1). $$
Formula \eqref{small gradient on the simplices} holds for every $ x\in\textnormal{relint}\left(\Delta_k
\right) $. By 0-homogeneity of $ \nabla_e\psi $, this extends to every $ x $ in the interior of $ C_k $.
Since $ k\in\{1,\ldots,m\} $ was arbitrary, using \eqref{connection between the gradients} we have that
$$ \|\nabla\psi(x)\|<C\varepsilon, $$
for $ H^{n-1}$-a.e. $ x\in \sfe $.

In particular, we have proved that for every $ i\in\mathbb{N} $ we can find a piecewise linear function
$ f_i\in\mathcal{L}(\sfe) $ such that
$$ \|u-f_i\|_\infty<\frac{2\sqrt{n}}{i} $$
and
$$ \|\nabla u(x)-\nabla f_i(x)\|<\frac{C}{i}, $$
for $ H^{n-1}$-a.e. $ x\in \sfe $. Therefore, $ f_i\rightarrow u $ uniformly on $ \sfe $ and
$ \nabla f_i\rightarrow\nabla u $ $ H^{n-1}$-a.e. in $ \sfe $. Besides, for every $ i\in\mathbb{N} $
and $ H^{n-1}$-a.e. $ x\in \sfe $ we have
$$ \|\nabla f_i(x)\|<\frac{C}{i}+\|\nabla u(x)\|\leq C+\max_{\sfe}\|\nabla u\|, $$
so that $ f_i\xrightarrow[\tau]{}u $.
\end{proof}
\subsection{Approximation of Lipschitz functions by $ C^1 $ functions}
Functions in $ C^1(\sfe) $ can in turn be used to $ \tau-$approximate Lipschitz functions
defined on the sphere.
\begin{lem}\label{approximation of lipschitz functions}
Let $ u\in\Lip $. Then there exists a sequence $ \{u_i\}\subseteq C^1(\sfe) $ such that
$ u_i\xrightarrow[\tau]{}u $.
\end{lem}

\begin{proof}
Let $ u\in\Lip $ be a Lipschitz function with Lipschitz constant $ L $. As stated in Theorem
\ref{McShane}, such a function can be extended to a map $ u:\mathbb{R}^n\rightarrow\mathbb{R} $
defined by
$$ 
\bar{u}(x)=\max_{z\in \sfe}[u(z)-L\|x-z\|], 
$$
for $ x\in\mathbb{R}^n $, which is still Lipschitz continuous with the same Lipschitz constant $ L $.
To simplify the notations, the extension will still be denoted by the same symbol $ u $.

Consider the annuli
$$ \mathcal{C}_0=\left\{x\in\mathbb{R}^n:\frac{1}{3}\leq\|x\|\leq\frac{5}{3}\right\} $$
and
$$ \mathcal{C}_1=\left\{x\in\mathbb{R}^n:\frac{2}{3}\leq\|x\|\leq\frac{4}{3}\right\},$$
and let $ \eta:\mathbb{R}^n\rightarrow[0,1] $ be a $ C^\infty $ function with $ \textnormal{supp}(\eta)
\subseteq\mathcal{C}_0 $ such that $ \eta\equiv 1 $ on $ \mathcal{C}_1 $.

The function $ \tilde{u}:=u\cdot\eta $ is Lipschitz continuous; let $ \widetilde{L} $ be its Lipschitz
constant.

We will use mollifiers to build our approximating sequence. Consider $ \varphi:\mathbb{R}^n
\rightarrow[0,+\infty) $ defined, for $ z\in\mathbb{R}^n $, by
$$ \varphi(z)=\begin{cases}c\cdot e^{\frac{1}{\|z\|^2-1}}\;\;\mbox{ if }\|z\|<1,\\
                                  0\;\;\;\;\;\;\;\;\;\;\;\;\;\;\;\,\mbox{ if }\|z\|\geq 1,\end{cases} $$
where $ c>0 $ is a constant such that
$$ \int_{\mathbb{R}^n}\varphi(z)dz=1. $$
The function $ \varphi $ is $ C^\infty $ with support $ \overline{B_1(0)} $. For $ \varepsilon>0 $, we
set
$$ \varphi_\varepsilon(z)=\frac{1}{\varepsilon^n}\varphi\left(\frac{z}{\varepsilon}\right),\;z\in
\mathbb{R}^n. $$
Note that, for every $ \varepsilon>0 $, $ \varphi_\varepsilon\in C^\infty(\mathbb{R}^n) $, its
support is the set $ \overline{B_\varepsilon(0)} $ and the following properties hold:
\begin{itemize}
\item $ \int_{\mathbb{R}^n}\varphi_\varepsilon(z)dz=1 $;
\item $ \lim_{\varepsilon\rightarrow 0^+}\int_{\mathbb{R}^n\backslash B_\delta(0)}
\varphi_\varepsilon(z)dz=0 $, for every $ \delta>0 $.
\end{itemize}

We now consider the $ C^\infty $ functions
$$ \widetilde{u}_\varepsilon(x)=\varphi_\varepsilon\ast\tilde{u}(x)=\int_{\mathbb{R}^n}\tilde{u}(x-y)
\varphi_\varepsilon(y)dy,\;x\in\mathbb{R}^n. $$

For $ x,y\in\mathbb{R}^n $, we have
$$ \vert\widetilde{u}_\varepsilon(x)-\widetilde{u}_\varepsilon(y)\vert=\Bigg|\int_{\mathbb{R}^n}
\varphi_\varepsilon(z)[\tilde{u}(x-z)-\tilde{u}(y-z)]dz\Bigg|\leq\widetilde{L}\|x-y\|. $$
Remembering \eqref{bound on the gradient of a lipschitz function}, this yields
\begin{equation}\label{the gradients are uniformly bounded}
\|\nabla\widetilde{u}_\varepsilon(x)\|\leq\sqrt{n}\cdot\widetilde{L},
\end{equation}
for every $ x\in\sfe $ and $ \varepsilon>0 $.

We also have that $ \widetilde{u}_\varepsilon\rightarrow u $ uniformly on $ \sfe $, as
$ \varepsilon\rightarrow 0 $. Indeed, since $ \tilde{u} $ is uniformly continuous, for a fixed
$ \varepsilon'>0 $ there is a $ \delta>0 $ such that
$$ \vert\tilde{u}(z_1)-\tilde{u}(z_2)\vert<\frac{\varepsilon'}{2}, $$
for every $ z_1,z_2\in\mathbb{R}^n $ satisfying $ \|z_1-z_2\|\leq\delta $. Therefore, for
$ \varepsilon>0 $ and $ x\in\sfe $ we get
\begin{eqnarray*}
\vert\widetilde{u}_\varepsilon(x)-u(x)\vert&=&\Bigg|\int_{\mathbb{R}^n}\tilde{u}(x-y)
\varphi_\varepsilon(y)dy-\int_{\mathbb{R}^n}\tilde{u}(x)\varphi_\varepsilon(y)dy\Bigg|\\
&\leq&\int_{\mathbb{R}^n}\varphi_\varepsilon(y)\vert\tilde{u}(x-y)-\tilde{u}(x)\vert dy
<\frac{\varepsilon'}{2}+\int_{\mathbb{R}^n\backslash B_\delta(0)}\varphi_\varepsilon(y)\vert 
\tilde{u}(x-y)-\tilde{u}(x)\vert dy. 
\end{eqnarray*}
Since
$$ \lim_{\varepsilon\rightarrow 0^+}\int_{\mathbb{R}^n\backslash B_\delta(0)}\varphi_\varepsilon(y)
dy=0, $$
there exists $ 0<\varepsilon_0<1 $ such that for every $ 0<\varepsilon<\varepsilon_0 $ we have
that, for $ x\in \sfe $,
\begin{eqnarray*}
\vert\widetilde{u}_\varepsilon(x)-u(x)\vert&<&\frac{\varepsilon'}{2}+\int_{\overline{B_1(0)}\backslash 
B_\delta(0)}\varphi_\varepsilon(y)\vert\tilde{u}(x-y)-\tilde{u}(x)\vert dy\\
&\leq&\frac{\varepsilon'}{2}+2M\int_{\overline{B_1(0)}\backslash B_\delta(0)}\varphi_\varepsilon(y)dy=
\frac{\varepsilon'}{2}+2M\int_{\mathbb{R}^n\backslash B_\delta(0)}\varphi_\varepsilon(y)dy<
\varepsilon',
\end{eqnarray*}
where
$$ M=\max_{z\in\overline{B_2(0)}}\vert\tilde{u}(z)\vert. $$
This proves that $ \widetilde{u}_\varepsilon\rightarrow u $ uniformly on $ \sfe $.

To show the $ H^{n-1}$-a.e. convergence of the gradients, for an arbitrary $ k\in\{1,\ldots,n\} $ we
evaluate, for $ x\in\mathbb{R}^n $,
\begin{equation}\notag
\begin{split}
\frac{\partial\widetilde{u}_\varepsilon}{\partial x_k}(x)&=\lim_{h\rightarrow 0}
\frac{\widetilde{u}_\varepsilon(x+he_k)-\widetilde{u}_\varepsilon(x)}{h}\\
&=\lim_{h\rightarrow 0}\int_{\mathbb{R}^n}\varphi_\varepsilon(y)\frac{\tilde{u}(x-y+he_k)-
\tilde{u}(x-y)}{h}dy\\
&=\int_{\mathbb{R}^n}\varphi_\varepsilon(y)\frac{\partial\tilde{u}}{\partial x_k}(x-y)dy=
\varphi_\varepsilon\ast\frac{\partial\tilde{u}}{\partial x_k}(x),
\end{split}
\end{equation}
where we have used the Dominated Convergence Theorem, which can be applied because of the
Lipschitz continuity of $ \tilde{u} $ and the Lebesgue integrability of $ \varphi_\varepsilon $. The
Lipschitz continuity of $ \tilde{u} $, together with the fact that $ \eta $ is compactly supported, also
implies $ \frac{\partial\tilde{u}}{\partial x_k}\in L^1\left(\mathbb{R}^n\right) $. The properties of the
$ \varphi_\varepsilon$'s then guarantee that
$$ \frac{\partial\widetilde{u}_\varepsilon}{\partial x_k}=\varphi_\varepsilon\ast\frac{\partial\tilde{u}}
{\partial x_k}\rightarrow\frac{\partial\tilde{u}}{\partial x_k}\mbox{ in }L^1\left(\mathbb{R}^n\right),
\mbox{ as }\varepsilon\rightarrow 0^+, $$
for $ k=1,\ldots,n $.

We now turn the family $ \{\widetilde{u}_\varepsilon\}_{\varepsilon>0} $ into a sequence
$ \{\widetilde{u}_i\}_{i\in\mathbb{N}} $ by choosing $ \varepsilon=1/i $, $ i\in\mathbb{N} $, and 
renaming $ \widetilde{u}_i:=\widetilde{u}_{1/i} $ for the sake of simplicity. Every sequence of functions
which converges in $ L^1 $ possesses a subsequence which converges a.e. to the same limit, hence
there is a subsequence
$$ \left\{\widetilde{u}_{i_j^{(1)}}\right\}_{j\in\mathbb{N}}\subseteq\{\widetilde{u}_i\}_{i\in
\mathbb{N}} $$
such that
$$ \frac{\partial\widetilde{u}_{i_j^{(1)}}}{\partial x_1}(x)\rightarrow\frac{\partial\tilde{u}}{\partial x_1}(x)
\mbox{ as }j\rightarrow\infty, $$
for a.e. $ x\in\mathbb{R}^n $, and
$$ \frac{\partial\widetilde{u}_{i_j^{(1)}}}{\partial x_k}\rightarrow\frac{\partial\tilde{u}}{\partial x_k}
\mbox{ in }L^1\left(\mathbb{R}^n\right)\mbox{ as }j\rightarrow\infty, $$
for $ k=2,\ldots,n $. We repeat the process for every $ 2\leq j\leq n $, finding sequences
$$ \left\{\widetilde{u}_{i_j^{(n)}}\right\}_{j\in\mathbb{N}}\subseteq\left\{\widetilde{u}_{i_j^{(n-1)}}
\right\}_{j\in\mathbb{N}}\subseteq\ldots\subseteq\left\{\widetilde{u}_{i_j^{(1)}}\right\}_{j\in\mathbb{N}}
\subseteq\{\widetilde{u}_i\}_{i\in\mathbb{N}}. $$
If we set $ u_j:=\widetilde{u}_{i_j^{(n)}} $, $ j\in\mathbb{N} $, we have that
$$ \frac{\partial u_j}{\partial x_k}(x)\rightarrow\frac{\partial\tilde{u}}{\partial x_k}(x)\mbox{ as }
j\rightarrow\infty, $$
for a.e. $ x\in\mathbb{R}^n $ and for all $ k=1,\ldots,n $. This implies
\begin{equation}\label{a.e. convergence of the gradients}
\nabla u_j(x)\rightarrow\nabla u(x)\mbox{ as }j\rightarrow\infty,
\end{equation}
for $ H^{n-1}$-a.e. $ x\in \sfe $, using \eqref{connection between the gradients} again.

Recalling that $ \{u_j\}_{j\in\mathbb{N}} \subseteq\{\widetilde{u}_\varepsilon\}_{\varepsilon>0} $, from 
the fact that $ \widetilde{u}_\varepsilon\rightarrow u $ uniformly on $ \sfe $ and the properties
\eqref{the gradients are uniformly bounded}, \eqref{a.e. convergence of the gradients} we conclude
that $ u_j\xrightarrow[\tau]{} u $.
\end{proof}
We have actually proved that $ C^\infty(\sfe) $ is $ \tau-$dense in $ \Lip $. However, for our
purposes Lemma \ref{approximation of lipschitz functions} will be sufficient.
\subsection{Density of $ \mathcal{L}(\sfe) $ in $ \Lip $}
Putting things together, we obtain the following density result.
\begin{lem}\label{density result}
The space $ \mathcal{L}(\sfe) $ is $ \tau-$dense in $ \Lip $.
\end{lem}
\begin{proof}
We have already noted that $ \mathcal{L}(\sfe)\subseteq\Lip $.

Let $ u\in\Lip $ and let $ U\subseteq\Lip $ be an open neighbourhood of $ u $ (with respect to
$ \tau $). Because of Lemma \ref{approximation of lipschitz functions}, $ U $ contains a function
$ v\in C^1(\sfe) $. Moreover, since $ U $ is also an open neighbourhood of $ v $, from Lemma
\ref{approximation of C^1 functions} we have that there is $ f\in\mathcal{L}(\sfe) $ such that
$ f\in U $.
\end{proof}
The tools developed throughout this section allow us now to prove Proposition
\ref{valuations are determined by their values on support functions}, which is the key ingredient
for the proof of Theorem \ref{characterization result}.
\begin{proof}[Proof of Proposition \ref{valuations are determined by their values on support functions}]
Let $ \mu_1,\mu_2:\Lip\rightarrow\mathbb{R} $ be continuous valuations such that $ \mu_1=\mu_2 $
on $ \mathcal{H}(\sfe) $. Then $ \mu_1=\mu_2 $ on $ \widehat{\mathcal{H}}(\sfe) $, as can be
proved by induction, using the valuation property and Lemma
\ref{maximum and minimum of support functions}. From Lemma
\ref{piecewise linear functions are minima of support functions} we obtain that $ \mu_1=\mu_2 $ on
$ \mathcal{L}(\sfe) $ too.

Let now $ u\in\Lip $. From Lemma \ref{density result}, there exists a sequence $ \{f_i\}\subseteq
\mathcal{L}(\sfe) $ such that $ f_i\xrightarrow[\tau]{}u $. Since $ \mu_1 $, $ \mu_2 $ are continuous,
$$ \mu_1(u)=\lim_{i\rightarrow\infty}\mu_1(f_i)=\lim_{i\rightarrow\infty}\mu_2(f_i)=\mu_2(u), $$
hence the conclusion.
\end{proof}

\section{The homogeneous decomposition for continuous and dot product invariant valuations}\label{section The homogeneous decomposition}

This section is devoted to the proof of the following result.

\begin{thm}\label{McMullen on Lip}
Let $ \mu:\Lip\rightarrow\mathbb{R} $ be a continuous and dot product invariant valuation. Then
there exist continuous and dot product invariant valuations $ \mu_0,\ldots,\mu_n:\Lip\rightarrow
\mathbb{R} $ such that $ \mu_i $ is $ i-$homogeneous, for $ i=0,\ldots,n $, and
\begin{equation}\label{McMullen decomposition formula on Lip}
\mu(\lambda u)=\sum_{i=0}^n\lambda^i\mu_i(u),
\end{equation}
for every $ u\in\Lip $ and $ \lambda>0 $.
\end{thm}

\begin{proof}
Let $ \mu:\Lip\rightarrow\mathbb{R} $ be as in the hypothesis. Consider the map $ \nu:
\mathcal{K}^n\rightarrow\mathbb{R} $ defined by $ \nu(K)=\mu(h_K) $, for $ K\in\mathcal{K}^n $;
because of Lemma \ref{moving valuations}, this is a valuation on ${\mathcal K}^n$ which is translation invariant and
continuous with respect to the Hausdorff metric. From Theorem \ref{McMullen} we obtain continuous
and translation invariant valuations $ \nu_0,\ldots,\nu_n:\mathcal{K}^n\rightarrow\mathbb{R} $ such
that each $ \nu_i $ is $ i-$homogeneous and \eqref{McMullen decomposition formula} holds.

Define now $ \mu_i:\mathcal{H}(\sfe)\rightarrow\mathbb{R} $, $ i=0,\ldots,n $, by setting
$ \mu_i(h_K)=\nu_i(K) $, for $ h_K\in\mathcal{H}(\sfe) $. Reading McMullen's formula
\eqref{McMullen decomposition formula} in the support functions' setting, we have that for every
$ h_K\in\mathcal{H}(\sfe) $ and $ \lambda>0 $
\begin{equation}\label{decomposition formula on H}
\mu(\lambda h_K)=\mu(h_{\lambda K})=\nu(\lambda K)=\sum_{i=0}^n\lambda^i\nu_i(K)=
\sum_{i=0}^n\lambda^i\mu_i(h_K).
\end{equation}
This is the desired decomposition formula stated for support functions; we would now like to extend
it to all Lipschitz functions $ u\in\Lip $. To do that, we must first extend each $ \mu_i $ to $ \Lip $.

In what follows, for simplicity we write $h$ to denote a generic element of ${\mathcal H}(\sfe)$. We write
\eqref{decomposition formula on H} for an arbitrary $h$ and for $\lambda=k=1,\dots,n+1$:
\begin{equation}\label{McMullen system}
\mu(kh)=\sum_{i=0}^nk^i\mu_i(h).
\end{equation}
We see it as a system of $ n+1 $ equations in the $ n+1 $ unknowns $ \mu_0(h),
\mu_1(h),\ldots,\mu_n(h) $. The matrix associated with this system is
$$ 
M=\left(\begin{matrix} 1         & 1         & 1           & \cdots & 1\\
                                         1         & 2         & 2^2       & \cdots & 2^n\\
                                         \vdots & \vdots & \vdots   &            & \vdots\\
                                         1         & n         & n^2       & \cdots & n^n\\
                                         1         & n+1    & (n+1)^2 & \cdots & (n+1)^n
\end{matrix}\right), 
$$
which is a Vandermonde matrix, hence
$$ 
\det M=\prod_{1\leq i<j\leq n+1}(j-i)\neq 0, 
$$
and then $ M $ is nonsingular. Therefore, the system \eqref{McMullen system} is invertible and we can
find coefficients $ a_{ij} $, $ i=0,\ldots,n $, $ j=1,\ldots,n+1 $, such that
$$ \mu_i(h)=\sum_{j=1}^{n+1}a_{ij}\mu(jh); $$
note that the coefficients are independent of $h$. 
This allows us to extend the $ \mu_i$'s to $ \Lip $: for $ i=0,\ldots,n $ we set
\begin{equation}\label{second extension of mu_i}
\mu_i(u):=\sum_{j=1}^{n+1}a_{ij}\mu(ju),
\end{equation}
for every $ u\in\Lip $. We observe that for every $j\in\{1,\dots, n+1\}$, the function defined on $\Lip$ by $u\mapsto\mu(ju)$ inherits all the properties of $\mu$, i.e., it is a continuous and dot product invariant valuation on $\Lip$ as well. 

Let 
$$
\bar\mu=\sum_{i=0}^{n}\mu_i.
$$
By \eqref{decomposition formula on H}, $\mu$ and $\bar\mu$ coincide on 
${\mathcal H}(\sfe)$; hence, by Proposition \ref{valuations are determined by their values on support functions}, 
$\mu=\bar\mu$ on $\Lip$. 

It remains to be shown that, for every $ i\in\{0,\ldots,n\} $, $\mu_i$ is $i$-homogeneous on $\Lip$.
Fix $ i\in\{0,\ldots,n\} $. Let $\lambda>0$, and define $\mu',\,\mu''\colon\Lip\to\R$ by
$$
\mu'(u)=\mu_i(\lambda u),\quad \mu''(u)=\lambda^i\mu_i(u),\quad\forall u\in\Lip.
$$
These are continuous valuations and as they coincide on ${\mathcal H}(\sfe)$, they coincide on $\Lip$. This proves 
that $\mu_i$ is $i$-homogeneous. 
\end{proof}

\section{Characterisation of dot product and rotation invariant valuations}\label{The characterization result}

\subsection{Homogeneity and valuations on $ \Lip $}
The proof of Theorem \ref{characterization result} requires this last result.
\begin{prop}\label{homogeneous valuations on Lip}
Let $ n\geq 3 $ and $ 3\leq k\leq n $. Let $ \mu:\Lip\rightarrow\mathbb{R} $ be a continuous,
rotation invariant, dot product invariant and $ k-$homogeneous valuation. Then $ \mu\equiv 0 $ on
$ \Lip $.
\end{prop}

\begin{rmk}
This proposition shows that a significant number of valuations, namely the intrinsic volumes with homogeneity degree larger or equal to three, cannot be extended from the space of support functions to the wider set of Lipschitz functions. In particular the volume functional cannot be extended to $\Lip$, in dimension three and higher.  
\end{rmk}

To ease the reading, we have stated some of the steps of the proof of Proposition
\ref{homogeneous valuations on Lip} as lemmas. Their proofs are provided along the way.
\begin{proof}
Let $ n $, $ k $ and $ \mu $ be as in the hypothesis. Define $ \nu:\mathcal{K}^n\rightarrow
\mathbb{R} $ by setting
$$ \nu(K)=\mu(h_K), $$
for $ K\in\mathcal{K}^n $. The functional $ \nu $ is a $ k-$homogeneous, translation and rotation
invariant valuation which is continuous with respect to the Hausdorff metric, thanks to Lemma
\ref{moving valuations}. From Theorem \ref{Hadwiger}, we have that there exists a constant $ c\in
\mathbb{R} $ such that
$$ \mu(h_K)=\nu(K)=cV_k(K), $$
for every $ K\in\mathcal{K}^n $, where $ V_k $ is the $ k^{\textnormal{th}} $ intrinsic volume.

If $ c=0 $, then $ \mu=0 $ on $ \mathcal{H}(\sfe) $, and from Proposition
\ref{valuations are determined by their values on support functions} we have the assertion.

Suppose now $ c\neq 0 $. We will show that this leads to a contradiction. Since the functional
$ \frac{1}{c}\mu $ retains all of $ \mu$'s properties, up to dividing by $ c $ we can assume that
\begin{equation}\label{mu on support functions}
\mu(h_K)=V_k(K),
\end{equation}
for every $ K\in\mathcal{K}^n $.

For $ x\in\mathbb{R}^n $ we write $ x=(\xi,\eta) $, with $ \xi\in\mathbb{R}^k $ and $ \eta\in
\mathbb{R}^{n-k} $. Fix $ \overline{\xi}\in S^{k-1} $ and define $ u_{\overline{\xi}}:\mathbb{R}^n
\rightarrow\mathbb{R} $ by setting
$$ 
u_{\overline{\xi}}(x)=u_{\overline{\xi}}(\xi,\eta)=\|\xi-\langle\xi,\overline{\xi}\rangle
\overline{\xi}\|, 
$$
for $ x\in\mathbb{R}^n $. Consider the $ (k-1)-$dimensional disk in $ \mathbb{R}^n $ defined by
$$ 
D_{\overline{\xi}}=\left\{(\xi,0)\in\mathbb{R}^k\times\mathbb{R}^{n-k}:\langle\xi,\overline{\xi}
\rangle=0,\|\xi\|\leq 1\right\}. 
$$
The map $ u_{\overline{\xi}} $ is the support function of $ D_{\overline{\xi}} $. In fact, up to a change
of coordinate system, we can assume $ \overline{\xi}=(1,0,\ldots,0) $; from the definition of support
function, for every $ (\xi,\eta)\in\mathbb{R}^n $ we have
$$ 
h_{D_{\overline{\xi}}}(\xi,\eta)=\max_{(\xi',0)\in D_{\overline{\xi}}}\langle\xi,\xi'\rangle=
\max_{(\xi',0)\in D_{\overline{\xi}}}\langle(\xi_2,\ldots,\xi_k),(\xi_2',\ldots,\xi_k')\rangle=\|(\xi_2,
\ldots,\xi_k)\|=u_{\overline{\xi}}(\xi,\eta). 
$$

Define now $ v_{\overline{\xi}}:\mathbb{R}^n\rightarrow\mathbb{R} $ by setting
$$ 
v_{\overline{\xi}}(x)=v_{\overline{\xi}}(\xi,\eta)=\langle\xi,\overline{\xi}\rangle, 
$$
for $ x\in\mathbb{R}^n $; $ v_{\overline{\xi}} $ is the support function of the convex body (in fact, a singleton)
$ \{(\overline{\xi},0)\} $.

For $ \lambda\geq 1 $, consider $ w_{\lambda,\,\overline{\xi}}:\sfe\rightarrow\mathbb{R} $
defined by
$$ w_{\lambda,\,\overline{\xi}}=(\lambda u_{\overline{\xi}}-v_{\overline{\xi}})\wedge\mathbb{O}, $$
where $ \mathbb{O} $ denotes the function which is identically zero on $ \sfe $. Note that
$ w_{\lambda,\,\overline{\xi}}=h_{\lambda D_{\overline{\xi}}-(\overline{\xi},0)}\wedge\mathbb{O}
\in\Lip $, being a minimum of Lipschitz functions. Therefore, $ \mu $ can be evaluated at
$ w_{\lambda,\,\overline{\xi}} $, and we do that in the
following lemma.
\begin{lem}\label{values of mu}
We have
$$ \mu(w_{\lambda,\,\overline{\xi}})=-\frac{\omega_{k-1}}{k}\lambda^{k-1}, $$
where $ \omega_{k-1} $ denotes the measure of the unit ball of $ \mathbb{R}^{k-1} $.
\end{lem}
\begin{proof}
From the valuation property we get
\begin{eqnarray}\label{mu on w}
\mu(w_{\lambda,\,\overline{\xi}})&=&\mu((\lambda u_{\overline{\xi}}-v_{\overline{\xi}})\wedge
\mathbb{O})=\mu(\lambda u_{\overline{\xi}}-v_{\overline{\xi}})+\mu(\mathbb{O})-\mu((\lambda
u_{\overline{\xi}}-v_{\overline{\xi}})\vee\mathbb{O})\nonumber\\
&=&\mu(\lambda u_{\overline{\xi}}-v_{\overline{\xi}})-\mu((\lambda u_{\overline{\xi}}-
v_{\overline{\xi}})\vee\mathbb{O}),
\end{eqnarray}
since $ \mu(\mathbb{O})=0 $, because of the homogeneity.

As we have already pointed out, $ \lambda u_{\overline{\xi}}-v_{\overline{\xi}}=
h_{\lambda D_{\overline{\xi}}-(\overline{\xi},0)} $, and remembering
\eqref{mu on support functions} we obtain
\begin{equation}\label{mu on lambda u-v}
\mu(\lambda u_{\overline{\xi}}-v_{\overline{\xi}})=V_k(\lambda D_{\overline{\xi}}-
(\overline{\xi},0))=V_k(\lambda D_{\overline{\xi}})=\lambda^kV_k(D_{\overline{\xi}})=0,
\end{equation}
where the last equality follows from the fact that $ D_{\overline{\xi}} $ has dimension $ k-1 $.

Now, $ (\lambda u_{\overline{\xi}}-v_{\overline{\xi}})\vee\mathbb{O} $ is the support function
of $ \conv\left((\lambda D_{\overline{\xi}}-(\overline{\xi},0))\cup\{0\}\right) $ (see Lemma
\ref{maximum and minimum of support functions}), which is a cone with vertex at the origin, base
$ \lambda D_{\overline{\xi}}-(\overline{\xi},0) $ and height $ 1 $, since $ \|\overline{\xi}\|=1 $. From
\eqref{mu on w}, \eqref{mu on lambda u-v} and \eqref{mu on support functions} we get
$$ \mu(w_{\lambda,\,\overline{\xi}})=-\mu((\lambda u_{\overline{\xi}}-v_{\overline{\xi}})\vee
\mathbb{O})=-V_k\left(\conv\left((\lambda D_{\overline{\xi}}-(\overline{\xi},0))\cup\{0\}\right)
\right)=-\frac{\omega_{k-1}}{k}\lambda^{k-1}. $$
\end{proof}
The next lemma concerns the support set $ \supp(w_{\lambda,\,\overline{\xi}}) $ of the function
$ w_{\lambda,\,\overline{\xi}} $.
\begin{lem}
For every $ (\xi,0)\in\supp(w_{\lambda,\,\overline{\xi}}) $ we have
$$ \|\xi-\overline{\xi}\|<\frac{\sqrt{2}}{\lambda}. $$
\end{lem}
\begin{proof}
Like before, we assume $ \overline{\xi}=(1,0,\ldots,0) $. Thus, for every $ (\xi,\eta)\in \sfe $,
$$ w_{\lambda,\,\overline{\xi}}(\xi,\eta)=\left(\lambda\sqrt{\xi_2^2+\ldots+\xi_k^2}-\xi_1\right)
\wedge 0. $$
If $ (\xi,0)\in\supp(w_{\lambda,\,\overline{\xi}}) $, we have $ \|\xi\|=1 $ and $ \lambda\sqrt{\xi_2^2
+\ldots+\xi_k^2}-\xi_1\leq 0 $, hence
\begin{equation}\label{inequality}
\sqrt{\xi_2^2+\ldots+\xi_k^2}\leq\frac{\xi_1}{\lambda}.
\end{equation}
In particular, this implies $ \xi_1\geq 0 $.

We write $ \xi=(\xi_1,\xi') $, with $ \xi'\in\mathbb{R}^{k-1} $. Since $ \|\xi\|=1 $ and $ \xi_1\geq
0 $, we have $ \xi_1=\sqrt{1-\|\xi'\|^2} $. Using this last equality in \eqref{inequality} we obtain
$$ \|\xi'\|\leq\frac{\sqrt{1-\|\xi'\|^2}}{\lambda}, $$
which in turn gives
$$ \|\xi'\|^2\leq\frac{1}{1+\lambda^2}<\frac{1}{\lambda^2}. $$
We can also estimate
$$ \vert\xi_1-1\vert=1-\xi_1=1-\sqrt{1-\|\xi'\|^2}=\frac{\|\xi'\|^2}{1+\sqrt{1-\|\xi'\|^2}}
\leq\|\xi'\|^2<\frac{1}{\lambda^2}. $$

From these inequalities we get
$$ \|\xi-\overline{\xi}\|^2=\|\xi-(1,0,\ldots,0)\|^2=\vert\xi_1-1\vert^2+\|\xi'\|^2<
\frac{1}{\lambda^4}+\frac{1}{\lambda^2}\leq\frac{2}{\lambda^2}, $$
since $ \lambda\geq 1 $. The assertion follows.
\end{proof}
This result yields the following one.
\begin{lem}
For every $ \xi_1,\xi_2\in S^{k-1} $ such that $ \|\xi_1-\xi_2\|\geq\frac{4}{\lambda} $ we have
$$ w_{\lambda,\,\xi_1}\cdot w_{\lambda,\,\xi_2}=\mathbb{O}. $$
\end{lem}
\begin{proof}
Take $ \xi_1,\xi_2 $ as in the hypothesis. Suppose the result to be false. Then there is a point
$ (\widetilde{\xi},\widetilde{\eta})\in \sfe $ such that
$$ w_{\lambda,\,\xi_1}(\widetilde{\xi},\widetilde{\eta})\cdot w_{\lambda,\,\xi_2}(\widetilde{\xi},
\widetilde{\eta})\neq 0. $$
Note that $ w_{\lambda,\,\xi_1}(0,\widetilde{\eta})=w_{\lambda,\,\xi_2}(0,\widetilde{\eta})=0 $, hence
$ \widetilde{\xi}\neq 0 $.

For $ i=1,2 $, the function
$$ w_{\lambda,\,\xi_i}(\xi,\eta)=\left[\lambda\|\xi-\langle\xi,\xi_i\rangle\xi_i\|-\langle\xi,\xi_i
\rangle\right]\wedge 0  $$
is $ 1-$homogeneous with respect to $ \xi $, and since $  w_{\lambda,\,\xi_i}(\widetilde{\xi},
\widetilde{\eta})\neq 0 $, we also have $ w_{\lambda,\,\xi_i}(\widehat{\xi},\widetilde{\eta})\neq 0 $, 
where $ \widehat{\xi}=\frac{\widetilde{\xi}}{\|\widetilde{\xi}\|} $. This means that $ (\widehat{\xi},
\widetilde{\eta})\in\supp(w_{\lambda,\,\xi_i}) $, hence $ (\widehat{\xi},0)\in
\supp(w_{\lambda,\,\xi_i}) $ too (since $ w_{\lambda,\,\xi_i} $ does not depend on $ \eta $), and from
the previous lemma we have
$$ \|\widehat{\xi}-\xi_i\|<\frac{\sqrt{2}}{\lambda}, $$
for $ i=1,2 $. Therefore,
$$ \|\xi_1-\xi_2\|\leq\|\xi_1-\widehat{\xi}\|+\|\widehat{\xi}-\xi_2\|<\frac{2\sqrt{2}}{\lambda}, $$
which contradicts the hypothesis.
\end{proof}
Iterating, the previous result can be extended to any finite number of points.

\begin{cor}\label{distant points give orthogonal w's}
Let $ N\in\mathbb{N} $ and $ \xi_1,\ldots,\xi_N\in S^{k-1} $ be such that $ \|\xi_i-\xi_j\|\geq
\frac{4}{\lambda} $, for every $ i\neq j $. Then
$$ w_{\lambda,\,\xi_i}\cdot w_{\lambda,\,\xi_j}=\mathbb{O}, $$
for every $ i\neq j $.
\end{cor}
We will need a couple more results. The first one concerns the behaviour of a general valuation on
non-positive orthogonal functions.
\begin{lem}\label{mu on a minimum of functions}
Let $ N\in\mathbb{N} $ and $ u_1,\ldots,u_N\in\Lip $. If $ u_i\leq 0 $ for every $ i=1,\ldots,N $ and
$ u_i\cdot u_j=\mathbb{O} $ for $ i\neq j $, then
$$ \mu\left(\bigwedge_{i=1}^Nu_i\right)=\sum_{i=1}^N\mu(u_i). $$
\end{lem}

\begin{proof}
The set $ \left(\Lip,\vee,\wedge\right) $ is a lattice, and since every valuation on a lattice satisfies the
Inclusion-Exclusion Principle, we can write
$$ \mu\left(\bigwedge_{j=1}^Nu_j\right)=\sum_{1\leq j\leq N}\mu(u_j)-\sum_{1\leq j_1<j_2\leq N}
\mu(u_{j_1}\vee u_{j_2})+ $$
$$ +\sum_{1\leq j_1<j_2<j_3\leq N}\mu(u_{j_1}\vee u_{j_2}\vee u_{j_3})-\ldots+(-1)^{N-1}
\mu\left(\bigvee_{j=1}^Nu_j\right). $$
The hypotheses imply that $ u_{j_1}\vee\ldots\vee u_{j_m}=\mathbb{O} $, for every $ m\in\{1,
\ldots,N\} $ and $ \{j_1,\ldots,j_m\}\subseteq\{1,\ldots,N\} $. The conclusion immediately follows.
\end{proof}
The next well-known lemma allows us to find sufficiently many points on the unit sphere which are
not too close to each other.
\begin{lem}
Let $ N\in\mathbb{N} $, $ N\geq 2 $. For every $ \nu\in\mathbb{N} $ there are $ N_\nu=\nu^{N-1} $
points $ x_1,\ldots,x_{N_\nu}\in S^{N-1} $ such that
$$ \|x_i-x_j\|\geq\frac{1}{\sqrt{N}\nu}, $$
for $ i\neq j $.
\end{lem}
\begin{proof}
Fix $ N\in\mathbb{N} $, $ N\geq 2 $, and take $ \nu\in\mathbb{N} $. For $ a=(a_1,\ldots,a_{N-1}) $,
with $ a_1,\ldots,a_{N-1}\in\{0,1,\ldots,\nu-1\} $, we define
$$ 
x_a'=\frac{1}{\sqrt{N}}\left(\frac{a_1}{\nu},\ldots,\frac{a_{N-1}}{\nu}\right)\in\R^{N-1}. 
$$
These are $ \nu^{N-1} $ points, and they satisfy
$$ \|x_a'-x_b'\|\geq\frac{1}{\sqrt{N}\nu}, $$
for every $ a\neq b $. Moreover, $ \|x_a'\|<1 $ for every $ a $.

Consider now
$$ 
x_a=(x_a',\sqrt{1-\|x_a'\|^2})\in S^{N-1},
$$
for $ a=(a_1,\ldots,a_{N-1}) $ with $ a_1,\ldots,a_{N-1}\in\{0,1,\ldots,\nu-1\} $. These are
$ \nu^{N-1} $ points on the sphere, and we have
$$ \|x_a-x_b\|\geq\|x_a'-x_b'\|\geq\frac{1}{\sqrt{N}\nu}, $$
for every $ a\neq b $.
\end{proof}

We will now use these results to build a sequence of Lipschitz functions which will yield the
contradiction we are looking for. Choose $ N=k $ in the last lemma and take $ \nu\in\mathbb{N} $.
Then we have $ N_\nu=\nu^{k-1} $ points $ x_1,\ldots,x_{N_\nu}\in S^{k-1} $ such that
$$ \|x_i-x_j\|\geq\frac{1}{\sqrt{k}\nu}, $$
for every $ i\neq j $. Let
$$ \lambda_\nu=4\sqrt{k}\nu; $$
note that $ \lambda_\nu\geq 1 $. Since $ \frac{2k-2}{k}\geq\frac{4}{3}>1 $, we can pick a number
$$ 1<p<\frac{2k-2}{k} $$
and define the function $ \psi_\nu:\sfe\rightarrow\mathbb{R} $,
$$ \psi_\nu=\frac{1}{\nu^p}\bigwedge_{i=1}^{N_\nu}w_{\lambda_\nu,\,x_i}. $$

From the $ k-$homogeneity of $\mu$, 
Lemma \ref{mu on a minimum of functions} (which can be applied
because the fact that $ \|x_i-x_j\|\geq\frac{1}{\sqrt{k}\nu}=\frac{4}{\lambda_\nu} $ allows us to use
Corollary \ref{distant points give orthogonal w's}) and Lemma \ref{values of mu}, we get
$$ \mu(\psi_\nu)=\frac{1}{\nu^{kp}}\mu\left(\bigwedge_{i=1}^{N_\nu}w_{\lambda_\nu,\,x_i}\right)=
\frac{1}{\nu^{kp}}\sum_{i=1}^{N_\nu}\mu(w_{\lambda_\nu,\,x_i})=-\frac{1}{\nu^{kp}}
\frac{\omega_{k-1}}{k}\lambda_\nu^{k-1}N_\nu=-c_k\nu^{2k-2-kp}, $$
where
$$ c_k=4^{k-1}\omega_{k-1}k^{\frac{k-3}{2}}>0. $$
Given how $ p $ was chosen, $ 2k-2-kp>0 $, hence
\begin{equation}\label{mu goes to infinity}
\mu(\psi_\nu)\rightarrow-\infty
\end{equation}
as $ \nu\rightarrow\infty $.

We would now like to prove that $ \psi_\nu\xrightarrow[\tau]{}\mathbb{O} $, as $ \nu\rightarrow
\infty $.

For every $ i=1,\ldots,N_\nu $ and $ (\xi,\eta)\in \sfe $, from the triangular and Cauchy-Schwarz 
inequalities we have
\begin{equation}\notag
\begin{split}
\vert w_{\lambda_\nu,\,x_i}(\xi,\eta)\vert &=\vert[\lambda_\nu u_{x_i}(\xi,\eta)-v_{x_i}(\xi,\eta)]
\wedge 0\vert\leq\vert\lambda_\nu u_{x_i}(\xi,\eta)-v_{x_i}(\xi,\eta)\vert\\
&=\vert\lambda_\nu\|\xi-\langle\xi,x_i\rangle x_i\|-\langle\xi,x_i\rangle\vert\leq\lambda_\nu
\left(\|\xi\|+\|\xi\|\cdot\|x_i\|^2\right)+\\
&+\|\xi\|\cdot\|x_i\|=(2\lambda_\nu+1)\|\xi\|\leq 2\lambda_\nu+1,
\end{split}
\end{equation}
since $ x_i\in S^{k-1} $. This yields $ \|w_{\lambda_\nu,\,x_i}\|_\infty\leq 2\lambda_\nu+1$ for
every $ i=1,\ldots,N_\nu $, and consequently
$$ 
\|\psi_\nu\|_\infty\leq\frac{2\lambda_\nu+1}{\nu^p}=\frac{8\sqrt{k}}{\nu^{p-1}}+
\frac{1}{\nu^p}.
$$
Since $ p>1 $, this implies that $ \psi_\nu\rightarrow\mathbb{O} $ uniformly on $ \sfe $ as
$ \nu\rightarrow\infty $.

We now look for a uniform bound on $ L(\psi_\nu) $, the Lipschitz constant of $ \psi_\nu $.
For $ i\in\{1,\ldots,N_\nu\} $, consider $ \widetilde{w}_{\lambda_\nu,\,x_i}=\lambda_\nu u_{x_i}-
v_{x_i} $. For $ (\xi_1,\eta_1),(\xi_2,\eta_2)\in \sfe $,
\begin{eqnarray*}
\vert\widetilde{w}_{\lambda_\nu,\,x_i}(\xi_1,\eta_1)&-&\widetilde{w}_{\lambda_\nu,\,x_i}(\xi_2,\eta_2)
\vert\leq\lambda_\nu\vert u_{x_i}(\xi_1,\eta_1)-u_{x_i}(\xi_2,\eta_2)\vert+\vert v_{x_i}(\xi_1,\eta_1)-
v_{x_i}(\xi_2,\eta_2)\vert\\
&=&\lambda_\nu\vert\|\xi_1-\langle\xi_1,x_i\rangle x_i\|-\|\xi_2-
\langle\xi_2,x_i\rangle x_i\|\vert+\vert\langle\xi_1-\xi_2,x_i\rangle\vert\\
&\leq&\lambda_\nu\|\xi_1-\xi_2-\langle\xi_1-\xi_2,x_i\rangle x_i\|+\vert\langle\xi_1-\xi_2,x_i\rangle\vert\\
&\leq&\lambda_\nu(\|\xi_1-\xi_2\|+\|\xi_1-\xi_2\|\cdot\|x_i\|^2)+\|\xi_1-\xi_2\|\cdot\|x_i\|\\
&=&(2\lambda_\nu+1)\|\xi_1-\xi_2\| \\
&\leq&(2\lambda_\nu+1)\|(\xi_1,\eta_1)-(\xi_2,\eta_2)\|. 
\end{eqnarray*}
This yields
$$ L(w_{\lambda_\nu,\,x_i})\leq L(\widetilde{w}_{\lambda_\nu,\,x_i})\leq 2\lambda_\nu+1. $$

Therefore, recalling that the Lipschitz constant of a finite minimum of functions is at most the
maximum of the Lipschitz constants, we get
$$
L(\psi_\nu)=L\left(
\frac1{\nu^p}\bigwedge_{i=1}^{N_\nu}w_{\lambda_\nu,x_i}\right)\le
\frac1{\nu^p}\max\{L(w_{\lambda_\nu,x_i})\colon i=1,\ldots,N_\nu\}\le
\frac{8\sqrt k \nu+1}{\nu^p}\le\frac{8\sqrt k +1}{\nu^{p-1}}.
$$
This, together with \eqref{bound on the gradient of a lipschitz function}, implies that
$$ \|\nabla\psi_\nu(x)\|\leq\frac{\sqrt{n}(8\sqrt{k}+1)}{\nu^{p-1}}, $$
for every $ \nu\in\mathbb{N} $ and $ H^{n-1}$-a.e. $ x\in \sfe $. The last inequality both tells us
that $ \nabla\psi_\nu\rightarrow 0 $ $ H^{n-1}$-a.e. in $ \sfe $, as $ \nu\rightarrow\infty $, and
that $ \|\nabla\psi_\nu\| $ is uniformly bounded by
$$ C=\sqrt{n}(8\sqrt{k}+1). $$

Therefore, $ \psi_\nu\xrightarrow[\tau]{}\mathbb{O} $ as $ \nu\rightarrow\infty $. Since $ \mu $ is
continuous, this gives $ \mu(\psi_\nu)\rightarrow\mu(\mathbb{O})=0 $, which is in contradiction with
\eqref{mu goes to infinity}. This concludes the proof of Proposition
\ref{homogeneous valuations on Lip}.
\end{proof}
\subsection{Proof of the characterisation result}
We are finally ready to prove Theorem \ref{characterization result}.
\begin{proof}[Proof of Theorem \ref{characterization result}]
Assume the functional $ \mu:\Lip\rightarrow\mathbb{R} $ to be defined by
\eqref{representation formula} for some constants $ c_0,c_1,c_2\in\mathbb{R} $, and write
$$ \mu(u)=\int_{\sfe}F(u,\|\nabla u\|)dH^{n-1}(x), $$
for $ u\in\Lip $, where $ F:\mathbb{R}\times[0,+\infty)\rightarrow\mathbb{R} $ is the $ C^\infty $
function given by
$$ F(x,y)=\frac{c_0}{H^{n-1}(\sfe)}+c_1x+c_2[(n-1)x^2-y^2], $$
for every $ (x,y)\in\mathbb{R}\times[0,+\infty) $.

To prove that $ \mu $ is a valuation, we take $ u,v\in\Lip $ and compute
\begin{eqnarray}\label{mu is a valuation}
\mu(u\vee v)+\mu(u\wedge v)&=&\int_{\sfe}F(u\vee v,\|\nabla(u\vee v)\|)dH^{n-1}(x)+\\
&&+\int_{\sfe}F(u\wedge v,\|\nabla(u\wedge v)\|)dH^{n-1}(x)\nonumber\\
&=&\int_UF(u,\|\nabla(u\vee v)\|)dH^{n-1}(x)+\int_VF(v,\|\nabla(u\vee v)\|)dH^{n-1}(x)+\nonumber\\
&&+\int_EF(u\vee v,\|\nabla(u\vee v)\|)dH^{n-1}(x)+\int_UF(v,\|\nabla(u\wedge v)\|)dH^{n-1}(x)+\nonumber\\
&&+\int_VF(u,\|\nabla(u\wedge v)\|)dH^{n-1}(x)+\int_EF(u\wedge v,\|\nabla(u\wedge v)\|)dH^{n-1}(x),\nonumber
\end{eqnarray}
where
$$ 
U=\left\{x\in \sfe:\;u(x)>v(x)\right\},\;  V=\left\{x\in \sfe:\;u(x)<v(x)\right\},\; E=\left\{x\in \sfe:
\;u(x)=v(x)\right\}. 
$$

Let $x\in\sfe$ be such that $u$, $v$, $u\vee v$ and $u\wedge v$ are differentiable at $x$. Then clearly
$$ 
\nabla(u\vee v)(x)=\begin{cases}\nabla u(x)\mbox{ if }x\in U,\\
                                                        \nabla v(x)\mbox{ if }x\in V,\\
\end{cases} 
\quad\mbox{and}\quad
\nabla(u\wedge v)(x)=\begin{cases}\nabla v(x)\mbox{ if }x\in U,\\
                                                             \nabla u(x)\mbox{ if }x\in V.\\
\end{cases} 
$$
On the other hand, if $u(x)=v(x)$ it is not hard to prove (see also \cite{Pagnini-Master Thesis}) that
$$
\nabla u(x)=\nabla v(x)=\nabla (u\vee v)(x)=\nabla (u\wedge v)(x).
$$
Hence we can reassemble the integrals in \eqref{mu is a valuation} so that
$$ 
\mu(u\vee v)+\mu(u\wedge v)=\int_{\sfe}F(u,\|\nabla u\|)dH^{n-1}(x)+\int_{\sfe}F(v,
\|\nabla v\|)dH^{n-1}(x)=\mu(u)+\mu(v). 
$$

We now prove that $ \mu $ is continuous. Let $ \{u_i\}\subseteq\Lip $ be such that $ u_i
\xrightarrow[\tau]{}u\in\Lip $. Then $ \|u_i-u\|_\infty\rightarrow 0 $, hence there exists $ I\in
\mathbb{N} $ such that $ \|u_i\|_\infty<\|u\|_\infty+1 $ for every $ i>I $. Set
$$ M=\max\{\|u_1\|_\infty,\ldots,\|u_I\|_\infty,\|u\|_\infty+1\}. $$
Because of the $ \tau-$convergence, there is also a $ C>0 $ such that
$$ \left(u_i(x),\|\nabla u_i(x)\|\right)\in K:=[-M,M]\times[0,C], $$
for every $ i\in\mathbb{N} $ and $ H^{n-1}$-a.e. $ x\in \sfe $. Let $ D=\max_{K}\vert F\vert $,
thus $ F\left(u_i,\|\nabla u_i\|\right)=F\vert_K\left(u_i,\|\nabla u_i\|\right) $ is dominated by the
constant function $ D $, which is integrable on $ \sfe $ since the sphere has finite measure. From
the Dominated Convergence Theorem we get
$$ \mu(u)=\int_{\sfe}F(u,\|\nabla u\|)dH^{n-1}(x)=\lim_{i\rightarrow\infty}\int_{\sfe}F(u_i,
\|\nabla u_i\|)dH^{n-1}(x)=\lim_{i\rightarrow\infty}\mu(u_i). $$

For what concerns rotation invariance, we have that for every $ u\in\Lip $ and $ \varphi\in
\O(n) $
\begin{equation}\notag
\begin{split}
\|\nabla(u\circ\varphi)(x)\|&=\sqrt{\|\nabla_e(u\circ\varphi)(x)\|^2-[(u\circ\varphi)(x)]^2}=
\sqrt{\Big\|\left(D\varphi(x)\right)^T\nabla_e u(\varphi(x))\Big\|^2-u(\varphi(x))^2}\\
&=\sqrt{\|\nabla_e u(\varphi(x))\|^2-u(\varphi(x))^2}=\|\nabla u(\varphi(x))\|,
\end{split}
\end{equation}
for a.e. $ x\in\sfe $, where we have used \eqref{connection between the gradients} and the fact
that the matrix $ (D\varphi(x))^T $, being orthogonal, preserves the norm. Therefore,
\begin{equation}\notag
\begin{split}
\mu(u\circ\varphi)&=\int_{\sfe}F(u(\varphi(x)),\|\nabla(u\circ\varphi)(x)\|)dH^{n-1}(x)\\
&=\int_{\sfe}F(u(\varphi(x)),\|\nabla u(\varphi(x))\|)dH^{n-1}(x)=\mu(u),
\end{split}
\end{equation}
where we have applied the change of variables $ y=\varphi(x) $.

It remains to be seen that $ \mu $ is dot product invariant. This can be proved with a direct
computation, but it is easier to show it via a trick which also gives us the chance to recall how some
intrinsic volumes can be written, something that will be useful during the second part of the proof too.
It is known that (see for instance \cite{Schneider}) for every $ K\in\mathcal{K}^n $, the intrinsic volumes $ V_0 $, $ V_1 $,
$ V_2 $ can be expressed as follows:
\begin{equation}\label{V_0}
V_0(K)=1,
\end{equation}
\begin{equation}\label{V_1}
V_1(K)=\frac{1}{\omega_{n-1}}\int_{\sfe}h_KdH^{n-1}(x),
\end{equation}
where $ \omega_{n-1} $ is the $ (n-1)-$dimensional Hausdorff measure of the unit $ (n-1)-$ball,
and, if $ K\in C^{2,+} $,
\begin{equation}\label{V_2}
V_2(K)=\int_{\sfe}h_K\cdot\textnormal{tr}\left(M(h_K)\right)dH^{n-1}(x),
\end{equation}
where $ \textnormal{tr}\left(M(h_K)\right) $ denotes the trace of the matrix $ M(h_K) $ given by
$$ M(h_K)=h_K\cdot\textnormal{Id}_{n-1}+(h_K^{ij}), $$
$ (h_K^{ij}) $ being the $ (n-1)\times(n-1) $ matrix of the second covariant derivatives of $ h_K $ with respect to a local orthonormal frame on the sphere. From \eqref{V_2} we get
\begin{eqnarray}\label{rewritten V_2}
V_2(K)&=&\int_{\sfe}h_K[(n-1)h_K+\Delta h_K]dH^{n-1}(x)\\
&=&\int_{\sfe}[(n-1)h_K^2+ h_K\textnormal{div}\left(\nabla h_K\right)]dH^{n-1}(x)\nonumber\\
&=&\int_{\sfe}\left[(n-1)h_K^2-\|\nabla h_K\|^2\right]dH^{n-1}(x),\nonumber
\end{eqnarray}
where the last equality follows from the Divergence Theorem (here $ \Delta $ denotes the spherical
Laplace operator). Therefore,
\begin{equation}\label{mu on H}
\mu(h_K)=c_0V_0(K)+c_1\omega_{n-1}V_1(K)+c_2V_2(K),
\end{equation}
for every convex body $ K\in C^{2,+} $.

For $ x\in\mathbb{R}^n $, consider the functional $ \mu_x:\Lip\rightarrow\mathbb{R} $ defined by
$ \mu_x(u)=\mu(u+\langle\cdot,x\rangle) $, for $ u\in\Lip $. This is still a continuous valuation on
$ \Lip $ and, because of \eqref{mu on H}, it satisfies
\begin{eqnarray*}
\mu_x(h_K)&=&\mu(h_K+\langle\cdot,x\rangle)\\
&=&\mu(h_{K+x})=c_0V_0(K+x)+c_1\omega_{n-1}V_1(K+x)+c_2V_2(K+x)\\
&=&c_0V_0(K)+c_1\omega_{n-1}V_1(K)+c_2V_2(K)=\mu(h_K),
\end{eqnarray*}
for every convex body $ K\in C^{2,+} $, since the intrinsic volumes are translation invariant.

Now, the integral in \eqref{V_2} only makes sense for support functions of $ C^{2,+} $ bodies, but its
rewritten form \eqref{rewritten V_2} is well defined for every support function $ h_K\in
\mathcal{H}(\sfe) $. Since $ C^{2,+} $ bodies are dense in $ \mathcal{K}^n $ with respect to the
Hausdorff metric, for an arbitrary $ h_K\in\mathcal{H}(\sfe) $ we can find a sequence $ \{h_{K_i}\}
\subseteq\mathcal{H}(\sfe) $ with $ \{K_i\}\subseteq C^{2,+} $ such that $ \|h_{K_i}-h_K\|_\infty
\rightarrow 0 $. Then we also have $ h_{K_i}\xrightarrow[\tau]{}h_K $ (see the proof of Lemma
\ref{continuity of a valuation on H}), and since $ \mu_x $ and $ \mu $ are continuous with respect
to $ \tau $ we get $ \mu_x(h_K)=\mu(h_K) $. From Proposition
\ref{valuations are determined by their values on support functions} it follows that they coincide on the
whole space $ \Lip $, hence $ \mu $ is dot product invariant.

Vice versa, let $ \mu:\Lip\rightarrow\mathbb{R} $ be a continuous, rotation invariant and dot
product invariant valuation. As we previously did, let us consider $ \nu:\mathcal{K}^n\rightarrow
\mathbb{R} $ defined by
$$ \nu(K)=\mu(h_K), $$
for $ K\in\mathcal{K}^n $, which is a translation and rotation invariant valuation that is continuous
with respect to the Hausdorff metric, because of Lemma \ref{moving valuations}. From Theorem
\ref{Hadwiger}, there are real constants $ c_0,c_1,\ldots,c_n $ such that
\begin{equation}\label{application of Hadwiger's theorem}
\mu(h_K)=\nu(K)=\sum_{i=0}^nc_iV_i(K),
\end{equation}
for every $ K\in\mathcal{K}^n $.

From Theorem \ref{McMullen on Lip}, there exist continuous and dot product invariant valuations
$$ 
\mu_0,\mu_1,\ldots,\mu_n:\Lip\rightarrow\mathbb{R} 
$$ 
such that $ \mu_i $ is $ i-$homogeneous, for $ i=0,1,\ldots,n $, and
$$ \mu(\lambda u)=\sum_{i=0}^n\lambda^i\mu_i(u), $$
for every $ \lambda>0 $ and $ u\in\Lip $. Moreover, if we go back to \eqref{second extension of mu_i} 
we deduce that the $ \mu_i$'s are rotation invariant too, since $ \mu $ is. Applying Proposition
\ref{homogeneous valuations on Lip} to $ \mu_i $, for $ i=3,\ldots,n $, we get
\begin{equation}\label{application of McMullen's theorem}
\mu(\lambda u)=\mu_0(u)+\lambda\mu_1(u)+\lambda^2\mu_2(u),
\end{equation}
for every $ \lambda>0 $ and $ u\in\Lip $.

Combining \eqref{application of Hadwiger's theorem} and \eqref{application of McMullen's theorem}
we have that, for every $ \lambda>0 $ and $ K\in\mathcal{K}^n $,
$$ \mu_0(h_K)+\lambda\mu_1(h_K)+\lambda^2\mu_2(h_K)=\mu(\lambda h_K)=\mu(h_{\lambda K})=
\sum_{i=0}^nc_iV_i(\lambda K)=\sum_{i=0}^nc_i\lambda^iV_i(K), $$
where the last equality follows from the $ i-$homogeneity of the $ i^{\textnormal{th}} $ intrinsic
volume. This implies $ \mu_0(h_K)=c_0V_0(K) $, $ \mu_1(h_K)=c_1V_1(K) $, $ \mu_2(h_K)=
c_2V_2(K) $ and $ c_3=\ldots=c_n=0 $. Therefore, taking $ \lambda=1 $, $ u=h_K $ in
\eqref{application of McMullen's theorem} and remembering \eqref{V_0}, \eqref{V_1},
\eqref{rewritten V_2}, we find
\begin{equation}
\mu(h_K)=c_0+c_1\int_{\sfe}h_KdH^{n-1}(x)+c_2\int_{\sfe}\left[(n-1)h_K^2-\|\nabla h_K\|^2
\right]dH^{n-1}(x),
\end{equation}
for every $ K\in C^{2,+} $, where we have renamed $ c_1:=c_1/\omega_{n-1} $. From the
first part of the proof, the functional $ \tilde{\mu}:\Lip\rightarrow\mathbb{R} $ defined by
$$ \tilde{\mu}(u)=c_0+c_1\int_{\sfe}udH^{n-1}(x)+c_2\int_{\sfe}\left[(n-1)u^2-\|\nabla u\|^2
\right]dH^{n-1}(x), $$
for $ u\in\Lip $, is a continuous valuation like $ \mu $, and they coincide on the set of support
functions of $ C^{2,+} $ bodies, hence on $ \mathcal{H}(\sfe) $, by density. We conclude the proof from
Proposition \ref{valuations are determined by their values on support functions}.
\end{proof}

\section{An improved version of Theorem \ref{McMullen on Lip}}

In this final section we refine Theorem \ref{McMullen on Lip} as follows. 

\begin{thm}\label{McMullen on Lip improved}
Let $n\ge3$ and $ \mu:\Lip\rightarrow\mathbb{R} $ be a continuous and dot product invariant
valuation. Then there exist continuous and dot product invariant valuations $ \mu_0,\ldots,\mu_{n-1}:
\Lip\rightarrow\mathbb{R} $ such that $ \mu_i $ is $ i-$homogeneous, for $ i=0,\ldots,n-1 $, and
\begin{equation}\label{McMullen decomposition formula on Lip}
\mu(\lambda u)=\sum_{i=0}^{n-1}\lambda^i\mu_i(u),
\end{equation}
for every $ u\in\Lip $ and $ \lambda>0 $.
\end{thm}

For the proof we will need the following result (see \cite{Schneider}, Theorem 6.4.8).

\begin{thm}\label{volume theorem}
Let $ \nu:\mathcal{K}^n\rightarrow\mathbb{R} $ be a continuous and translation invariant valuation
which is homogeneous of degree $ n $. Then there exists $ c\in\mathbb{R} $ such that $ \nu(K)=c
V_n(K) $, for every $ K\in\mathcal{K}^n $.
\end{thm}

\bigskip

\noindent
{\em Proof of Theorem \ref{McMullen on Lip improved}.}
We use the notations introduced in the proof of Theorem \ref{McMullen on Lip}; by the latter result
we only need to prove that $ \mu_n\equiv 0 $. By Theorem \ref{volume theorem}, there exists
$ c\in\mathbb{R} $ such that $ \nu_n(K)=cV_n(K) $, for every $ K\in\mathcal{K}^n $. In particular,
$ \nu_n $ is rotation invariant, hence $ \mu_n $ is rotation invariant on $ \mathcal{H}(\sfe) $.

Let us prove that $ \mu_n $ is rotation invariant on the whole space $ \Lip $. For a fixed $ \varphi\in
\mathbf{O}(n) $, consider $ \mu_n^\varphi:\Lip\rightarrow\mathbb{R} $ defined by
$$ \mu_n^\varphi(u)=\mu_n(u\circ\varphi)-\mu_n(u), $$
for $ u\in\Lip $. Such functional is a continuous valuation on $ \Lip $; because $ \mu_n $ is rotation
invariant on $ \mathcal{H}(\sfe) $, $ \mu_n^\varphi=0 $ on $ \mathcal{H}(\sfe) $. From
Proposition \ref{valuations are determined by their values on support functions}, $ \mu_n^\varphi=0 $
on $ \Lip $, so that $ \mu_n(u\circ\varphi)=\mu_n(u) $, for every $ u\in\Lip $ and
$ \varphi\in\mathbf{O}(n) $. Therefore, $ \mu_n $ is a continuous, rotation invariant, dot product
invariant and $ n-$homogeneous valuation on $ \Lip $, hence $ \mu_n\equiv 0 $, thanks to
Proposition \ref{homogeneous valuations on Lip}.
\begin{flushright}
$\square$
\end{flushright}

\end{document}